\documentclass{siamart171218}
\usepackage{hyperref}
\usepackage{epsfig}
\usepackage{amsfonts}
\usepackage{algorithm}
\usepackage{algorithmic}
\usepackage{graphicx}
\usepackage{subfig}
\usepackage{array}
\usepackage{enumerate}
\usepackage[T1]{fontenc}

\usepackage{pgfplots,tikz}
\usetikzlibrary{plotmarks}

\usepackage{todonotes}

\newcommand {\cL}  {\mathcal{L}}
\newcommand {\cH}  {\mathcal{H}}

\newcommand {\cV}  {\mathcal{V}}
\newcommand {\cW}  {\mathcal{W}}

\DeclareMathOperator{\Real}{Re}

\newcommand\Item[1][]{%
  \ifx\relax#1\relax  \item \else \item[#1] \fi
  \abovedisplayskip=0pt\abovedisplayshortskip=0pt~\vspace*{-1.25\baselineskip}}

\newcommand {\C}        {{\mathbb{C}}}
\newcommand {\R}        {{\mathbb{R}}}

\newcommand {\Hinf}	{{\mathcal H}_{\infty}}
\newcommand {\ri}	{{\mathrm i}}
\newcommand{\sn}{\mathsf{n}}
\newcommand{\sm}{\mathsf{m}}
\renewcommand{\sp}{\mathsf{p}}
\newcommand{\sd}{\mathsf{d}}

\newtheorem{example}[theorem]{Example}

\newtheorem{remark}[theorem]{Remark}
\newtheorem{assumption}[theorem]{Assumption}

\title{A Subspace Framework for ${\mathcal H}_\infty$-Norm Minimization}
\author{
Nicat Aliyev\thanks{E-Mail: \texttt{naliyev@ku.edu.tr}.} \and
Peter Benner\thanks{Max Planck Institute for Dynamics of Complex Technical Systems, Sandtorstra{\ss}e 1, 39106 Magdeburg, Germany, E-Mail: \texttt{benner@mpi-magdeburg.mpg.de}.} \and
Emre Mengi\thanks{Ko\c{c} University, Department of Mathematics, Rumeli Feneri Yolu 34450, Sar{\i}yer, Istanbul, Turkey, E-Mail: \texttt{emengi@ku.edu.tr}.} \and
Matthias Voigt\thanks{Corresponding author, Universit\"at Hamburg, Fachbereich Mathematik, Bereich Optimierung und Approxi\-mation, Bundesstra{\ss}e~55, 20146 Hamburg, Germany, E-Mail: \texttt{matthias.voigt@uni-hamburg.de} and Techni\-sche Universit\"at Berlin, Institut f\"ur Mathematik, Stra{\ss}e des 17. Juni 136, 10623 Berlin, Germany, E-Mail: \texttt{mvoigt@math.tu-berlin.de}.}
}

\begin{document}
\maketitle

\begin{abstract}
We deal with the minimization of the ${\mathcal H}_\infty$-norm of the transfer function of a
parameter-dependent descriptor system over the set of admissible parameter values. Subspace frameworks are proposed for such minimization problems where the involved systems are of large order. The proposed
algorithms are greedy interpolatory approaches inspired by our recent work 
[Aliyev et al., SIAM J. Matrix Anal. Appl., 38(4):1496--1516, 2017] for the 
computation of the ${\mathcal H}_\infty$-norm. In this work, we minimize the ${\mathcal H}_\infty$-norm 
of a reduced-order parameter-dependent system obtained by two-sided restrictions
onto certain subspaces. Then we expand the subspaces so that Hermite interpolation
properties hold between the full and reduced-order system at the optimal
parameter value for the reduced order system. We formally establish the superlinear
convergence of the subspace frameworks 
under some smoothness assumptions. The fast convergence
of the proposed frameworks in practice is illustrated by several large-scale systems. 
\end{abstract}



\section{Introduction}
In this work we are concerned with the minimization of the $\Hinf$-norm of a parameter-dependent descriptor system of the form
\begin{equation}\label{eq:desc_sys}
  \begin{split}
    \tfrac{\mathrm{d}}{\mathrm{d}t}E(\mu) x(\mu;t) &= A(\mu) x(\mu;t) + B(\mu) u(\mu;t), \\
                                          y(\mu;t) &= C(\mu) x(\mu;t).
  \end{split}
\end{equation}
Here, for an open and bounded set $\Omega \subseteq {\mathbb R}^{\sd}$, $E,\,A: \Omega \rightarrow {\mathbb R}^{\sn\times \sn}$, 
$B: \Omega \rightarrow {\mathbb R}^{\sn \times \sm}$, $C: \Omega \rightarrow {\mathbb R}^{\sp \times \sn}$ are matrix-valued 
functions defined by 
\begin{equation}\label{eq:matrix_value}
	\begin{split}
		E(\mu) & := f_1(\mu) E_1 + \ldots + f_{\kappa_E}(\mu) E_{\kappa_E}, \\
		A(\mu) & := g_1(\mu) A_1 + \ldots + g_{\kappa_A}(\mu) A_{\kappa_A}, \\
		B(\mu) & := h_1(\mu) B_1 + \ldots + h_{\kappa_B}(\mu) B_{\kappa_B}, \\
		C(\mu) & := k_1(\mu) C_1 + \ldots + k_{\kappa_C}(\mu) C_{\kappa_C},
	\end{split}
\end{equation}
for given matrices $E_1,\,\ldots,\,E_{\kappa_E},\,A_1,\,\ldots,\,A_{\kappa_A} \in {\mathbb R}^{\sn\times \sn}$, 
$B_1,\,\ldots,\,B_{\kappa_B} \in {\mathbb R}^{\sn\times \sm}$, $C_1,\,\ldots,\,C_{\kappa_C} \in {\mathbb R}^{\sp \times \sn}$, 
and real-analytic functions $f_1,\,\ldots,\,f_{\kappa_E},\,g_1\,\ldots,\,g_{\kappa_A},\,h_1,\,\ldots,\,h_{\kappa_B}$,
$k_1,\,\ldots,\,k_{\kappa_C}: \Omega \rightarrow {\mathbb R}$. The functions $x(\mu;\cdot): \R \to \R^\sn$, $u(\mu;\cdot):\R \to \R^\sm$, 
and $y(\mu;\cdot):\R \to \R^\sp$ are called (generalized) state, input, and output, respectively. 
If for a \emph{fixed} $\mu \in \Omega$, the matrix pencil $s E(\mu) - A(\mu)$ is \emph{regular} (that is, there exists a 
$\lambda \in \C$ with $\det(\lambda E(\mu) - A(\mu)) \neq 0$), we define the transfer function of \eqref{eq:desc_sys} by
\[
	H[\mu](s) := C(\mu) D(\mu,s)^{-1} B(\mu)\quad \text{with} \quad D(\mu,s) := s E(\mu) - A(\mu).
\]
For a \emph{fixed} $\mu$, the function $H[\mu](s)$ is real-rational in the indeterminate $s$, consequently, we use the notation $H[\mu](s) \in \mathbb{R}(s)^{\sp \times \sm}$. 
Observe that, since $H[\mu]$ is rational, it is analytic almost everywhere in $\C$. 

We define the following normed spaces of real-rational functions:
\begin{align*}
 \cL_\infty^{\sp \times \sm} &:= \left\{ H(s) \in \R(s)^{\sp \times \sm} \; \bigg| \; \sup_{\omega \in \R} \left\| H(\ri\omega) \right\|_2 < \infty \right\}, \\
 \cH_\infty^{\sp \times \sm} &:= \left\{ H(s) \in \R(s)^{\sp \times \sm} \; \bigg| \; \sup_{\lambda \in \C^+} \left\| H(\lambda) \right\|_2 < \infty \right\},
\end{align*}
where $\C^+ := \left\{ \lambda \in \C \; | \; \Real(\lambda) > 0 \right\}$.
For $H \in \cL_\infty^{\sp \times \sm}$, the $\cL_\infty$-norm is defined by
\begin{equation*}
 \left\| H \right\|_{\cL_\infty} := \sup_{\omega \in \R} \left\| H(\ri\omega) \right\|_2 = \sup_{\omega \in \R} \sigma(H(\ri \omega)),
\end{equation*}
where $\sigma(\cdot)$ denotes the largest singular value of its matrix argument. We assume throughout this text that the functions
under consideration are in $\cH_\infty^{\sp \times \sm}$. For such a function $H \in \cH_\infty^{\sp \times \sm}$, by employing the maximum principle for analytic functions, one can show that the $\cH_\infty$-norm is equivalent to the $\cL_\infty$-norm, that is
\begin{equation*}
 \left\| H \right\|_{\cH_\infty} := \sup_{s \in \C^+} \left\| H(s) \right\|_2 
 				= 
	\sup_{s \in \partial\C^+} \left\| H(s) \right\|_2 = \sup_{\omega \in \R} \sigma(H(\ri \omega)). 
\end{equation*}

In this work, we consider the problem of minimizing the $\cH_\infty$-norm of $H[\mu]$ over $\mu$ that belongs to a compact subset $\underline{\Omega}$ of $\Omega$, but keeping the assumption that $H[\mu] \in \cH_\infty^{\sp \times \sm}$ for every $\mu \in \underline{\Omega}$. The latter assumption holds for all of the examples that we consider later in this paper; most of these examples arise from real applications. 
Formally, we aim to determine $\mu_\ast \in \underline{\Omega}$ such that
\begin{equation*}
 \left\| H[{\mu_\ast}] \right\|_{\cH_\infty}  =  \min_{\mu \in \underline{\Omega}} \left\|H[\mu]\right\|_{\cH_{\infty}}.
\end{equation*}

Minimizing the $\cH_\infty$-norm of a parameter-dependent system is an important task in control engineering. 
For example, the parameter vector $\mu$ may consist of the design variables of a feedback controller. Then it is desirable to design an optimal $\cH_\infty$-controller that minimizes the influence of a noisy input signal to the regulated output, which corresponds to minimizing the $\cH_\infty$-norm of a closed-loop (parameter-dependent) transfer function, see, e.\,g., \cite{ZhoDG96} and the references therein. Note that in the latter application, it is normally further imposed that the controller stabilizes the closed-loop system. This condition does not play a prominent role here, but efficient stability checks would be needed for controller design. Other applications for $\cH_\infty$-norm minimization arise in the optimization of dynamic flow networks \cite{JohWSJC17}, parameter identification \cite{VizMPL16}, and model reduction \cite{VarP01}.

We focus on the large-scale setting, that is when $\sn$ is large. We additionally impose the condition that the numbers of inputs and outputs are relatively small, i.\,e., $\sn \gg \sm,\,\sp$. Here we present subspace frameworks that are inspired by our previous work \cite{Aliyev2017}. The proposed frameworks converge fast with respect to the subspace dimension. We provide a theoretical analysis which explains 
this convergence behavior and confirm our theoretical findings in practice by means of several numerical experiments.

\textbf{Outline.} The subspace frameworks are formally introduced in the next section. We first provide a 
basic greedy framework for ${\mathcal H}_\infty$-norm minimization in Algorithm~\ref{alg1}.
This framework reduces the order of the full-order system by employing two-sided restrictions to certain subspaces. It performs the ${\mathcal H}_\infty$-norm minimization on the reduced system,
then expands the restriction subspaces so that Hermite interpolation properties hold
between the full and reduced-order system at the optimal parameter value
for the reduced system. An extension of the basic framework is proposed in 
Algorithm~\ref{alg2}. There Hermite interpolation properties do not only hold at the optimal parameter value
for the reduced system, but also at nearby points. In Section \ref{sec:ROC}, we formally show 
that the basic subspace framework when there is only one parameter, and the extended framework
converge with a superlinear rate under some smoothness assumptions at the minimizer.
The performance of proposed basic subspace framework and its rate of convergence are illustrated for several examples in Section \ref{sec:experiments}.
As we report in the end, with the proposed subspace frameworks, only a few seconds are required
for the minimization of the ${\mathcal H}_{\infty}$-norm of a parameter-dependent system of order $10^4$,
in contrast to an approach that does not make use of reductions. 

\section{Subspace Frameworks}\label{sec:SF_define}
To deal with the large-scale problems described in the introduction, we employ two-sided restrictions in the flavor of the practice we followed for large-scale $\Hinf$-norm computation 
in \cite{Aliyev2017}. We choose two subspaces $\cV,\,\cW \subseteq \C^\sn$ of the same dimension, 
as well as matrices $V,\,W \in \C^{\sn \times \mathsf{k}}$ whose columns form orthonormal bases for these subspaces,
and define the reduced problem by
\begin{equation*}
  \begin{split}
		E^{V,W}(\mu) & := f_1(\mu) W^* E_1 V + \ldots + f_{\kappa_E}(\mu) W^* E_{\kappa_E} V, \\
		A^{V,W}(\mu) & := g_1(\mu) W^* A_1 V + \ldots + g_{\kappa_A}(\mu) W^* A_{\kappa_A} V, \\
		B^W(\mu) & := h_1(\mu) W^* B_1 + \ldots + h_{\kappa_B}(\mu) W^* B_{\kappa_B}, \\
		C^V(\mu) & := k_1(\mu) C_1 V + \ldots + k_{\kappa_C}(\mu) C_{\kappa_C} V.
	\end{split}
\end{equation*}
Associated with this system, there is the reduced transfer function
\[
   H^{\cV,\cW}[\mu](s):=C^{V}(\mu) D^{V,W}(\mu,s)^{-1} B^{W}(\mu) \quad \text{with} \quad D^{V,W}(\mu,s) := s E^{V,W}(\mu)-A^{V,W}(\mu)
\]
which turns out to be independent of the particular choice of the bases for $\cV$ and $\cW$.
Our subspace frameworks are based on the repeated minimization of $\left\| H^{{\mathcal V}, {\mathcal W}}(\mu) \right\|_{{\mathcal H}_\infty}$ for appropriate choices of the subspaces ${\mathcal V},\, {\mathcal W}$. 

\begin{algorithm}[tb]
 \begin{algorithmic}[1]
 
\REQUIRE{Matrices $E_1,\,\dots,\,E_{\kappa_E} \in \mathbb{R}^{\sn\times \sn}$,
 $A_1,\,\dots,\,A_{\kappa_A} \in \mathbb{R}^{\sn\times \sn}$, $B_1,\,\dots,\,B_{\kappa_B} \in \mathbb{R}^{\sn\times \sm}$, $C_1,\,\dots,\,C_{\kappa_C} \in \mathbb{R}^{\sp\times \sn}$ and functions 
 $f_1,\,\ldots,\,f_{\kappa_E},\,g_1,\,\ldots,\,g_{\kappa_A},\,h_1,\,\ldots,\,h_{\kappa_B},\, k_1,\,\ldots,\,k_{\kappa_C}$ as in \eqref{eq:matrix_value}.}
\ENSURE{Sequences $\big\{ \mu^{(k)} \big\}$, $\big\{ \omega^{(k)} \big\}$.}
\STATE Choose initial subspace ${\mathcal V}_0,\, {\mathcal W}_0 \subseteq \C^\sn$.
\FOR{$k = 1,\,2,\,\dots$}
	\STATE  $\mu^{(k)} \gets  \arg \min_{\mu \in \underline{\Omega}} \left\| H^{{\mathcal V}_{k-1}, {\mathcal W}_{k-1}} [\mu] \right\|_{{\mathcal H}_\infty}$. \label{line:opt_mu}
	\STATE  $\omega^{(k)} \gets \arg \max_{\omega \in {\mathbb R}\cup\{\infty\}} \sigma\left(\mu^{(k)},\omega\right)$. \label{line:opt_omega}
	\IF{$\sm = \sp$} \label{line1beg}
		\STATE $\widetilde{V}_k \gets D\big(\mu^{(k)},\mathrm{i}\omega^{(k)}\big)^{-1}B\big(\mu^{(k)}\big)$.
		\STATE $\widetilde{W}_k \gets D\big(\mu^{(k)},\mathrm{i}\omega^{(k)}\big)^{-\ast}C\big(\mu^{(k)}\big)^\ast$.
	\ELSIF{$\sm < \sp$}
		\STATE $\widetilde{V}_k \gets D\big(\mu^{(k)},\mathrm{i}\omega^{(k)}\big)^{-1}B\big(\mu^{(k)}\big)$.
		\STATE $\widetilde{W}_k \gets D\big(\mu^{(k)},\mathrm{i}\omega^{(k)}\big)^{-\ast}C\big(\mu^{(k)}\big)^\ast H\big[\mu^{(k)}\big]\big(\ri \omega^{(k)}\big)$.
	\ELSE
		\STATE $\widetilde{V}_k \gets D\big(\mu^{(k)},\mathrm{i}\omega^{(k)}\big)^{-1}B\big(\mu^{(k)}\big) H\big[\mu^{(k)}\big]\big(\ri \omega^{(k)}\big)^\ast$
		\STATE $\widetilde{W}_k \gets D\big(\mu^{(k)},\mathrm{i}\omega^{(k)}\big)^{-\ast}C\big(\mu^{(k)}\big)^\ast$.
	\ENDIF \label{line1end}
	\STATE ${\mathcal V}_k \gets {\mathcal V}_{k-1} \oplus {\rm Col}\big( \widetilde{V}_k \big)$ and 	${\mathcal W}_k \gets {\mathcal W}_{k-1} \oplus {\rm Col}\big(\widetilde{W}_k\big)$.
\ENDFOR
\end{algorithmic}
\caption{The basic greedy algorithm for $\Hinf$-norm minimization}
\label{alg1}
\end{algorithm}


The basic greedy framework is given in Algorithm~\ref{alg1} and throughout the restof this work, we use the short-hand notations
\[
	 \sigma(\mu,\omega) := \sigma (H[\mu](\ri \omega)) \quad \text{and} \quad \sigma^{{\mathcal V}, {\mathcal W}}(\mu,\omega) := \sigma\left( H^{{\mathcal V}, {\mathcal W}}[\mu](\ri \omega) \right).
\]
We will also make frequent use of certain partial derivatives of these functions, where we denote the variables that we differentiate by subscripts, e.\,g., $\sigma_\omega(\cdot,\cdot)$ denotes the first partial derivative with respect to the argument $\omega$, whereas $\sigma_\mu(\cdot,\cdot)$ denotes the gradient with respect to $\mu$. 
Additionally, we reserve the notations $\sigma_2(\mu,\omega)$ and $\sigma^{{\mathcal V}, {\mathcal W}}_2(\mu,\omega)$
for the second largest singular values of $H[\mu](\ri \omega)$ and $H^{{\mathcal V},{\mathcal W}}[\mu](\ri \omega)$, respectively.
At every iteration, the basic framework minimizes the $\Hinf$-norm of a reduced problem for a given pair of subspaces in line \ref{line:opt_mu}. Then it first computes an $\omega$ such that $\| H[\mu] \|_{{\mathcal H}_\infty} = \sigma(\omega, \mu)$ in line \ref{line:opt_omega} at the optimal $\mu$ value for the reduced problem, and expands the subspaces so that the following Hermite interpolation properties hold at the optimal $\mu$, $\omega$, which are immediate from \cite[Theorem~2.1]{Aliyev2017}, \cite[Theorem~1]{Gugercin2009}.
\begin{lemma}\label{thm:basic_int}
The following assertions hold regarding Algorithm~\ref{alg1} for each $j = 1,\,\ldots,\,k$:\vspace{1ex}
\begin{enumerate}
    \item[\bf (i)] It holds that $\left\| H\big[\mu^{(j)}\big] \right\|_{{\mathcal H}_\infty} = \sigma\big(\mu^{(j)},\omega^{(j)}\big) = \sigma^{ {\mathcal V}_k, {\mathcal W}_k } \big(\mu^{(j)}, \omega^{(j)}\big).$ \vspace*{1ex}
    \item[\bf (ii)] It holds that $\sigma_2\big(\mu^{(j)},\omega^{(j)}\big) = \sigma^{ {\mathcal V}_k, {\mathcal W}_k }_2 \big(\mu^{(j)}, \omega^{(j)}\big).$ \vspace*{1ex}
    \item[\bf (iii)] If the largest singular value $\sigma\big(\mu^{(j)},\omega^{(j)}\big)$ of $H[\mu^{(j)}](\ri \omega^{(j)})$ is simple, then
    \[
    	\nabla \big\| H\big[\mu^{(j)}\big] \big\|_{{\mathcal H}_\infty} = \sigma_\mu\big(\mu^{(j)},\omega^{(j)}\big) 
    				= \sigma^{ {\mathcal V}_k, {\mathcal W}_k }_\mu \big(\mu^{(j)},\omega^{(j)}\big).
	\]
     \item[\bf (iv)] We have $ 
     	\sigma_\omega \big(\mu^{(j)},\omega^{(j)}\big) = \sigma^{ {\mathcal V}_k,  {\mathcal W}_k}_\omega \big(\mu^{(j)},\omega^{(j)}\big) = 0.
	$
\end{enumerate}
\end{lemma}
Note that in part \textbf{(iv)} of the lemma above $\sigma_{\omega}\big(\mu^{(j)},\omega^{(j)}\big) = 0$ holds even if 
$\sigma\big(\mu^{(j)},\omega^{(j)}\big)$ is not simple, since $\omega^{(j)}$ is a maximizer of $\sigma\big(\mu^{(j)}, \cdot \big)$.
The equality $\sigma^{ {\mathcal V}_k,  {\mathcal W}_k}_\omega \big(\mu^{(j)},\omega^{(j)}\big) = 0$ follows
from the interpolation properties between $H[\mu](\ri \omega)$, $H^{{\mathcal V}_k, {\mathcal W}_k}[\mu](\ri \omega)$
and their first derivatives at $\mu = \mu^{(j)}, \: \omega = \omega^{(j)}$.

We also propose an extended version of the basic greedy framework in Algorithm~\ref{alg2}. 
For its description we define $e_{rq} := 1/\sqrt{2} (e_r + e_q)$ if $r \neq q$ and $e_{rr} :=  e_r$, where $e_r$ is the $r$-th column of the $\sd\times \sd$ identity matrix. The description may look complicated at first, but the only main difference is that it includes additional vectors in the subspaces in lines \ref{line:adjacent_expand_st}--\ref{line:adjacent_expand_end} to interpolate 
not only at the minimizers of the reduced problems, but also at nearby points. The motivation for the inclusion of these additional vectors is to draw a
theoretical conclusion about the accuracy of the second derivatives of the reduced singular value functions $\sigma^{{\mathcal V}_k, {\mathcal W}_k}(\cdot,\cdot)$
in approximating $\sigma(\cdot,\cdot)$ in the multivariate case. In practice, we observe that both Algorithm~\ref{alg1} and Algorithm~\ref{alg2} converge rapidly. But in the multivariate case, the inclusion of the additional vectors in the subspaces in Algorithm~\ref{alg2} makes its rate of convergence analysis neater. The interpolation properties of the extended framework are listed in the next result. Once again, these properties 
are immediate from \cite[Theorem~2.1]{Aliyev2017}.
\begin{lemma}\label{thm:extended_int}
The iterates $\big\{ \mu^{(k)} \big\}$, $\big\{ \omega^{(k)} \big\}$ by Algorithm~\ref{alg2} satisfy the assertions \textbf{(i)}--\textbf{(iii)} of Lemma~\ref{thm:basic_int} for each $j = 1,\,\dots,\,k$. Additionally, for each $j = 1,\, \dots,\, k$, $r = 1,\,\dots,\,\sd$, and $q = r,\,\dots,\,\sd$, we have the following: \vspace*{1ex}
\begin{enumerate}
\item[\bf (i)] It holds that $\left\| H\big[\mu^{(j,rq)}\big] \right\|_{{\mathcal H}_\infty} = \sigma\big(\mu^{(j,rq)},\omega^{(j,rq)}\big) = 	\sigma^{ {\mathcal V}_k,  {\mathcal W}_k } \big(\mu^{(j,rq)},\omega^{(j,rq)}\big)$. \vspace*{1ex}
\item[\bf (ii)] If the largest singular value $\sigma \big(\mu^{(j,rq)},\omega^{(j,rq)}\big)$ of $H[\mu^{(j,rq)}](\ri \omega^{(j,rq)})$ is simple, then
\[
	\nabla \big\| H\big[\mu^{(j,rq)}\big] \big\|_{{\mathcal H}_\infty} =  \sigma_\mu \big(\mu^{(j,rq)},\omega^{(j,rq)}\big) 
				= \sigma^{ {\mathcal V}_k, {\mathcal W}_k }_\mu \big(\mu^{(j,rq)},\omega^{(j,rq)}\big).
\]
\item[\bf (iii)] It holds that $
	\sigma_\omega \big(\mu^{(j,rq)},\omega^{(j,rq)}\big) =  \sigma^{ {\mathcal V}_k,  {\mathcal W}_k}_\omega \big(\mu^{(j,rq)},\omega^{(j,rq)}\big) = 0.
 	$
\end{enumerate}
\end{lemma}

\begin{algorithm}[tb]
 \begin{algorithmic}[1]
\REQUIRE{Matrices $E_1,\,\dots,\,E_{\kappa_E} \in \mathbb{R}^{\sn\times \sn}$,
 $A_1,\,\dots,\,A_{\kappa_A} \in \mathbb{R}^{\sn\times \sn}$, $B_1,\,\dots,\,B_{\kappa_B} \in \mathbb{R}^{\sn\times \sm}$, $C_1,\,\dots,\,C_{\kappa_C} \in \mathbb{R}^{\sp\times \sn}$ and functions 
 $f_1,\,\ldots,\,f_{\kappa_E},\,g_1,\,\ldots,\,g_{\kappa_A},\,h_1,\,\ldots,\,h_{\kappa_B},\, k_1,\,\ldots,\,k_{\kappa_C}$ as in \eqref{eq:matrix_value}.}
\ENSURE{Sequences $\big\{ \mu^{(k)} \big\}$, $\big\{ \omega^{(k)} \big\}$.}
\STATE Choose initial subspace ${\mathcal V}_0, {\mathcal W}_0 \subseteq \C^\sn$.
\FOR{$k = 1,\,2,\,\dots$}
\STATE  $\mu^{(k)} \gets  \arg \min_{\mu \in \underline{\Omega}} \left\| H^{{\mathcal V}_{k-1}, {\mathcal W}_{k-1}} [\mu] \right\|_{{\mathcal H}_\infty}$.
	\STATE  $\omega^{(k)} \gets \arg \max_{\omega \in {\mathbb R}\cup\{\infty\}} \sigma\left(\mu^{(k)},\omega\right)$.
	\IF{$\sm = \sp$} \label{line2beg}
		\STATE $\widetilde{V}_k \gets D\big(\mu^{(k)},\mathrm{i}\omega^{(k)}\big)^{-1}B\big(\mu^{(k)}\big)$.
		\STATE $\widetilde{W}_k \gets D\big(\mu^{(k)},\mathrm{i}\omega^{(k)}\big)^{-\ast}C\big(\mu^{(k)}\big)^\ast$.
	\ELSIF{$\sm < \sp$}
		\STATE $\widetilde{V}_k \gets D\big(\mu^{(k)},\mathrm{i}\omega^{(k)}\big)^{-1}B\big(\mu^{(k)}\big)$.
		\STATE $\widetilde{W}_k \gets D\big(\mu^{(k)},\mathrm{i}\omega^{(k)}\big)^{-\ast}C\big(\mu^{(k)}\big)^\ast H\big[\mu^{(k)}\big]\big(\ri \omega^{(k)}\big)$.
	\ELSE
		\STATE $\widetilde{V}_k \gets D\big(\mu^{(k)},\mathrm{i}\omega^{(k)}\big)^{-1}B\big(\mu^{(k)}\big) H\big[\mu^{(k)}\big]\big(\ri \omega^{(k)}\big)^\ast$.
		\STATE $\widetilde{W}_k \gets D\big(\mu^{(k)},\mathrm{i}\omega^{(k)}\big)^{-\ast}C\big(\mu^{(k)}\big)^\ast$.
	\ENDIF \label{line2end}
	\STATE ${\mathcal V}_k \gets {\mathcal V}_{k-1} \oplus {\rm Col}\big( \widetilde{V}_k \big)$ and ${\mathcal W}_k \gets {\mathcal W}_{k-1} \oplus {\rm Col}\big(\widetilde{W}_k\big)$.
	\IF{$k \geq 2$} \label{line:adjacent_expand_st}
	\STATE $h^{(k)} \gets \big\| \mu^{(k)} - \mu^{(k-1)} \big\|_2$. 
	\FOR{$r = 1,\,2,\, \dots,\, \sd$}
		\FOR{$q = r,\, \dots,\, \sd$} 
			\STATE $\mu^{(k,rq)} \gets \mu^{(k)} + h^{(k)} e_{rq}$.
			\STATE $\omega^{(k,rq)} \gets \arg \max_{\omega \in {\mathbb R} \cup \{\infty\}} \sigma\big(\mu^{(k,rq)}, \omega\big)$.
			\IF{$\sm = \sp$} \label{line3beg}
				\STATE $\widetilde{V}_k^{(rq)} \gets D\big(\mu^{(k,rq)},\mathrm{i}\omega^{(k,rq)}\big)^{-1}B\big(\mu^{(k,rq)}\big)$. 
				\STATE $\widetilde{W}_k^{(rq)} \gets D\big(\mu^{(k,rq)},\mathrm{i}\omega^{(k,rq)}\big)^{-\ast}C\big(\mu^{(k,rq)}\big)^\ast$.
			\ELSIF{$\sm < \sp$}
				\STATE $\widetilde{V}_k^{(rq)} \gets D\big(\mu^{(k,rq)},\mathrm{i}\omega^{(k,rq)}\big)^{-1}B\big(\mu^{(k,rq)}\big)$.
				\STATE $\widetilde{W}_k^{(rq)} \gets D\big(\mu^{(k,rq)},\mathrm{i}\omega^{(k,rq)}\big)^{-\ast}C\big(\mu^{(k,rq)}\big)^\ast H\big[\mu^{(k,rq)}\big]\big(\ri \omega^{(k,rq)}\big)$.
			\ELSE
				\STATE $\widetilde{V}_k^{(rq)} \gets D\big(\mu^{(k,rq)},\mathrm{i}\omega^{(k,rq)}\big)^{-1}B\big(\mu^{(k,rq)}\big)H\big[\mu^{(k,rq)}\big]\big(\ri \omega^{(k,rq)}\big)^\ast$. 
				\STATE $\widetilde{W}_k^{(rq)} \gets D\big(\mu^{(k,rq)},\mathrm{i}\omega^{(k,rq)}\big)^{-\ast}C\big(\mu^{(k,rq)}\big)^\ast$.
			\ENDIF \label{line3end}
		\STATE ${\mathcal V}_k \gets {\mathcal V}_{k}  \oplus {\rm Col}\big( \widetilde{V}_k^{(rq)} \big)$ and ${\mathcal W}_k \gets {\mathcal W}_{k}  \oplus {\rm Col}\big( \widetilde{W}_k^{(rq)} \big)$.
		\ENDFOR
	\ENDFOR
	\ENDIF \label{line:adjacent_expand_end}
\ENDFOR
\end{algorithmic}
\caption{The extended greedy algorithm for $\Hinf$-norm minimization}
\label{alg2}
\end{algorithm}

Before we start with the rate of convergence analysis, a few comments regarding the two algorithms are in order.
\begin{remark}
\begin{enumerate}
 \item[\bf (i)] The distinctions of cases in lines \ref{line1beg}--\ref{line1end} in Algorithm~\ref{alg1} and lines \ref{line2beg}--\ref{line2end} and \ref{line3beg}--\ref{line3end} are done such that the subspaces $\mathcal{V}_k$ and $\mathcal{W}_k$ have the same dimension. This is needed in order to obtain a \emph{regular} reduced matrix pencil $D^{V_k,W_k}\big(\mu^{(k)},s\big)$ and a well-defined reduced transfer function $H^{\mathcal{V}_k,\mathcal{W}_k}\big[\mu^{(k)}\big](s)$. In practice, a regularization procedure can be performed \cite{AntM07} to obtain a regular reduced matrix pencil, even if the above distinctions of cases are not carried out. In the above algorithms we make the silent assumption that the transfer functions $H^{\mathcal{V}_k,\mathcal{W}_k}\big[\mu^{(k)}\big](s)$ are well-defined and in $\cL_\infty^{\sp \times \sm}$ for all $k$. Note that the reduced dynamical systems associated with the transfer functions $H^{\mathcal{V}_k,\mathcal{W}_k}\big[\mu^{(k)}\big](s)$ are not necessarily asymptotically stable, so the transfer functions are not necessarily in $\cH_\infty^{\sp \times \sm}$. However, for the algorithm, the latter does not lead to any problem.
 \item[\bf (ii)] In this paper, we only consider parameter-dependent linear time-invariant systems. Efficient algorithms for the computation of the $\cL_\infty$-norm however, have also been recently considered for transfer functions of a more general class of systems \cite{Aliyev2017, SchV18}. The results presented here can be transferred to this more general situation without any changes in the algorithm description.
\end{enumerate}
\end{remark}

\section{Rate of Convergence Analysis}\label{sec:ROC}
In this section, we perform a rate of convergence analysis for Algorithms~\ref{alg1} and~\ref{alg2}.
We view $\mu^{(k+1)}, \mu^{(k)}, \mu^{(k-1)}$ as functions of $\mu^{(1)}$. 
Letting $\mu_\ast$ be a local or a global minimizer of $\| H[\cdot] \|_{{\mathcal H}_\infty}$,
we assume $\mu^{(k+1)},\, \mu^{(k)},\, \mu^{(k-1)} \rightarrow \mu_\ast$ as $\mu^{(1)} \rightarrow \mu_\ast$. 
Our main result is a superlinear convergence result, i.\,e., for all $k\ge2$ there exists a constant $C$, independent of $\mu^{(1)}$, such that
\[
	\big\| \mu^{(k+1)} - \mu_\ast \big\|_2   \leq  C \left(  \big\| \mu^{(k)} - \mu_\ast \big\|_2  \cdot
				\max \big\{ \big\| \mu^{(k)} - \mu_\ast \big\|_2, \big\| \mu^{(k-1)} - \mu_\ast \big\|_2 \big\} \right)
\]
for all $\mu^{(1)}$ sufficiently close to $\mu_\ast$.

The analysis here addresses the smooth setting, that is, throughout this section we
assume the following:
\begin{assumption}[Smoothness]\label{assume:smoothness}
\textbf{(i)} The 
supremum of $\sigma(\mu_\ast, \cdot)$ is attained uniquely, say at $\omega_\ast$,
and \textbf{(ii)} $\sigma(\mu_\ast, \omega_\ast) > 0$ is a simple singular value of 
$H[\mu_\ast](\ri \omega_*)$.
\end{assumption}

Many of the results in this section are established uniformly over every $\mu^{(1)}$ that is
sufficiently close to $\mu_\ast$.
The dependence of the reduced transfer function $H^{{\mathcal V}_k, {\mathcal W}_k}[\mu](s)$,
as well as the reduced singular value functions 
$\sigma^{{\mathcal V}_k,\, {\mathcal W}_k}(\cdot,\cdot), \sigma^{{\mathcal V}_k, {\mathcal W}_k}_2(\cdot,\cdot)$
on $\mu^{(1)}$ is implicitly given through the subspaces ${\mathcal V}_k,\, {\mathcal W}_k$ or equivalently,
the matrices $V_k,\, W_k$ whose columns form orthonormal bases for these subspaces.
We start with uniform Lipschitz continuity results for these functions with respect to $\mu^{(1)}$.
Note that in this result and in the subsequent discussions, $\sigma_{\min}(\cdot)$ denotes the smallest singular value
of its matrix argument, whereas 
\[
  \overline{{\mathcal B}}\big(\widetilde{\mu},\eta\big) := \big\{ \mu \in {\mathbb R}^\sd \; \big| \; \big\| \mu - \widetilde{\mu} \big\|_2 \leq \eta \big\} 
     \quad \text{and} \quad 
  \overline{{\mathcal B}}\big(\widetilde{\omega},\eta\big) := \big\{ \omega \in {\mathbb R} \; \big| \; \big|\omega - \widetilde{\omega} \big| \leq \eta \big\}
\] 
for given $\widetilde{\mu} \in {\mathbb R}^\sd$, $\widetilde{\omega} \in {\mathbb R}$, and $\eta > 0$.
\begin{lemma}[Uniform Lipschitz continuity]\label{lemma:Lipschitz_cont}
Suppose that, for some $\beta > 0$, the point $\mu^{(1)}$ is such that $\sigma_{\min}\big( D^{V_k, W_k}(\mu_\ast, {\rm i} \omega_\ast) \big) \geq \beta$.  
Then there exist constants $\eta_\mu$, $\eta_\omega$, $\gamma$ independent of $\mu^{(1)}$ such that \vspace*{1ex}
\begin{enumerate}
     \Item[\bf (i)] \begin{multline*}\left\| H^{{\mathcal V}_k, {\mathcal W}_k}\big[\widetilde{\mu}\big](\ri\omega) - H^{{\mathcal V}_k, {\mathcal W}_k}[\mu](\ri\omega) \right\|_2  \leq  \gamma \big\| \widetilde{\mu} - \mu \big\|_2
     \\ \forall\, \widetilde{\mu}, \,\mu \in \overline{{\mathcal B}}(\mu_\ast,\eta_\mu), \,\forall\, \omega \in \overline{{\mathcal B}}(\omega_\ast,\eta_\omega);\end{multline*} 
     \Item[\bf (ii)] \begin{multline*}\big\| H^{{\mathcal V}_k, {\mathcal W}_k}[\mu]\big(\ri\widetilde{\omega}\big) - H^{{\mathcal V}_k, {\mathcal W}_k}[\mu](\ri\omega) \big\|_2  \leq  \gamma \big| \widetilde{\omega} - \omega \big|
     \\ \forall\, \mu \in \overline{{\mathcal B}}(\mu_\ast,\eta_\mu),\, \forall\, \widetilde{\omega},\, \omega \in \overline{{\mathcal B}}(\omega_\ast,\eta_\omega);\end{multline*}
     \Item[\bf (iii)] \begin{multline*} \left| \sigma^{{\mathcal V}_k, {\mathcal W}_k} \big(\widetilde{\mu}, \omega\big) - \sigma^{{\mathcal V}_k, {\mathcal W}_k} (\mu, \omega) \right| \leq  \gamma \big\| \widetilde{\mu} - \mu \big\|_2,\\  \big| \sigma^{{\mathcal V}_k, {\mathcal W}_k}_2 \big(\widetilde{\mu}, \omega\big) - \sigma^{{\mathcal V}_k, {\mathcal W}_k}_2 (\mu, \omega) \big| \leq  \gamma \big\| \widetilde{\mu} - \mu \big\|_2 \\ \forall\, \widetilde{\mu},\, \mu \in \overline{{\mathcal B}}(\mu_\ast,\eta_\mu),\, \forall\, \omega \in \overline{{\mathcal B}}(\omega_\ast,\eta_\omega), \end{multline*}
      \Item[\bf (iv)] \begin{multline*} \big| \sigma^{{\mathcal V}_k, {\mathcal W}_k} \big(\mu, \widetilde{\omega}\big) - \sigma^{{\mathcal V}_k, {\mathcal W}_k} (\mu, \omega) \big| \leq  \gamma \big| \widetilde{\omega} - \omega \big|, \\
      \big| \sigma^{{\mathcal V}_k, {\mathcal W}_k}_2 \big(\mu, \widetilde{\omega}\big) - \sigma^{{\mathcal V}_k, {\mathcal W}_k}_2 (\mu, \omega) \big| \leq  \gamma \big| \widetilde{\omega} - \omega \big|
      \\ \forall\, \mu \in \overline{{\mathcal B}}(\mu_\ast,\eta_\mu),\, \forall\, \widetilde{\omega},\, \omega \in \overline{{\mathcal B}}(\omega_\ast,\eta_\omega).\end{multline*}
\end{enumerate}
\end{lemma}
\begin{proof} 
	
	By Weyl's theorem \cite[Theorem~4.3.1]{Horn1985}, for every $\mu \in \underline{\Omega}$ and $\omega \in {\mathbb R}$
	we have
	\begin{equation*}
	\begin{split}
	      \big| \sigma_{\min}\big(D^{V_k,W_k}(\mu,\omega)\big)  -  \sigma_{\min}\big(D^{V_k,W_k}(\mu_\ast,\omega_\ast)\big)  \big|
	           & \leq
	       \big\|     D^{V_k,W_k}(\mu,\omega)   -   D^{V_k,W_k}(\mu_\ast,\omega_\ast)  \big\|_2 \\
	             & = 
	        {\| W_k^\ast   ( D(\mu,\omega)  -  D(\mu_\ast, \omega_\ast) ) V_k  \|}_2 \\
	            & \leq
	        {\| D(\mu,\omega)  -  D(\mu_\ast, \omega_\ast) \|}_2 \\
	              & \leq
	       \nu( {\| \mu - \mu_\ast \|}_2 +  | \omega - \omega_\ast |) 
	\end{split}       
	\end{equation*}
	for some $\nu > 0$,
	where the last inequality is due to the fact that $D(\cdot,\cdot)$ is continuously differentiable in a neighborhood of 
	$\underline{\Omega} \times {\mathbb R}$. This uniform Lipschitz continuity property of 
	$\sigma_{\min}\big(D^{V_k,W_k}(\cdot,\cdot)\big)$, 
	combined with $\sigma_{\min}\big( D^{V_k, W_k}(\mu_\ast, {\rm i} \omega_\ast) \big) \geq \beta$,
	implies the existence of $\eta_\mu,\, \eta_\omega$ 
	independent of $\mu^{(1)}$ such that
	\[
		\sigma_{\min}\big(D^{V_k,W_k} (\mu,\omega)\big)   \geq \beta/2
		\quad \forall\, \mu \in \overline{{\mathcal B}}(\mu_\ast,\eta_\mu),\, \forall\, \omega \in \overline{{\mathcal B}}(\omega_\ast,\eta_\omega).
	\] 
	It follows that $(\mu, \omega) \mapsto H^{{\mathcal V}_k, {\mathcal W}_k}[\mu]({\rm i} \omega)$ is 
	differentiable $\forall\, \mu \in \overline{{\mathcal B}}(\mu_\ast,\eta_\mu),\, \forall\, \omega \in \overline{{\mathcal B}}(\omega_\ast,\eta_\omega)$.
\begin{enumerate}
	\item[\bf (i)] For every $\mu \in \overline{{\mathcal B}}(\mu_\ast,\eta_\mu)$ and $\omega \in \overline{{\mathcal B}}(\omega_\ast,\eta_\omega)$,
	by the product and chain rule we obtain
	\begin{equation}\label{eq:red_transfer_der}
		\begin{split}
			\frac{ \partial  H^{{\mathcal V}_k,{\mathcal W}_k}[\mu]({\rm i} \omega)  }{ \partial \mu_j }   = 
				\frac{\partial C^{V_k}(\mu)}{\partial \mu_j} D^{V_k, W_k}(\mu, {\rm i} \omega)^{-1} B^{W_k}(\mu)
						+  \hskip 22ex \\
				C^{V_k}(\mu) 
				D^{V_k, W_k}(\mu, {\rm i} \omega)^{-1} \frac{\partial D^{V_k, W_k}(\mu, {\rm i} \omega)}{ \partial \mu_j } 
				D^{V_k, W_k}(\mu, {\rm i} \omega)^{-1} B^{W_k}(\mu) +  \hskip 4ex \\
				C^{V_k}(\mu) D^{V_k, W_k}(\mu, {\rm i} \omega)^{-1} \frac{ \partial B^{W_k}(\mu) }{ \partial \mu_j }
		\end{split}
	\end{equation}
	for $j = 1,\, \dots,\, \sd$. Setting 
	\begin{equation*}
		\begin{split}
	 	M'_{D,j} & := \max \left\{ \left\| \frac{ \partial D(\mu, {\rm i} \omega) }{ \partial \mu_j } \right\|_2 \; \bigg| \; 
				\mu \in \overline{{\mathcal B}}(\mu_\ast,\eta_\mu), \omega \in \overline{{\mathcal B}}(\omega_\ast,\eta_\omega) \right\}	\\
		M'_{C,j} & := \max \left\{ \left\| \frac{ \partial C(\mu) }{ \partial \mu_j } \right\|_2  \; \bigg|  \;
				\mu \in \overline{{\mathcal B}}(\mu_\ast,\eta_\mu) \right\},	\: 
		M_{C} := \max \left\{ {\| C(\mu) \|}_2 \:  | \: 
				\mu \in \overline{{\mathcal B}}(\mu_\ast,\eta_\mu) \right\}, \\
		M'_{B,j}  & := \max \left\{ \left\| \frac{ \partial B(\mu) }{ \partial \mu_j } \right\|_2 \; \bigg| \;
				\mu \in \overline{{\mathcal B}}(\mu_\ast,\eta_\mu) \right\}, \:
		M_{B}  := \max \left\{ {\| B(\mu) \|}_2 \:  | \: 
				\mu \in \overline{{\mathcal B}}(\mu_\ast,\eta_\mu) \right\},
		\end{split}
	\end{equation*}
	and exploiting 
	\begin{multline*}
\left\| \frac{\partial D(\mu, {\rm i} \omega)}{\partial \mu_j} \right\|_2 \geq \left\| \frac{\partial D^{V_k, W_k}(\mu, {\rm i} \omega)}{ \partial \mu_j } \right\|_2, \; {\| C(\mu) \|}_2 \geq \big\| C^{V_k}(\mu) \big\|_2, \\ {\| B(\mu) \|}_2 \geq \big\| B^{W_k}(\mu) \big\|_2, \; 
\left\| \frac{ \partial C(\mu) }{ \partial \mu_j } \right\|_2 \geq \left\| \frac{ \partial C^{V_k}(\mu) }{ \partial \mu_j } \right\|_2, \;
\left\| \frac{ \partial B(\mu) }{ \partial \mu_j } \right\|_2 \geq \left\| \frac{ \partial B^{W_k}(\mu) }{ \partial \mu_j } \right\|_2,
	\end{multline*}
	as well as $\sigma_{\min}\big(D^{V_k,W_k} (\mu,\omega)\big) \geq \beta/2$, we deduce from \eqref{eq:red_transfer_der} that
	\[
			\left\|
			\frac{ \partial  H^{{\mathcal V}_k,{\mathcal W}_k}[\mu]({\rm i} \omega)  }{ \partial \mu_j }
			\right\|_2
					\leq
			2 \frac{M'_{C,j} M_{B}}{\beta} + 4 \frac{M_{C} M'_{D,j} M_{B}}{\beta^2}  +  2 \frac{M_{C} M'_{B,j}}{\beta} =: M_j
	\]
	for all $\mu \in \overline{{\mathcal B}}(\mu_\ast,\eta_\mu)$, $\omega \in \overline{{\mathcal B}}(\omega_\ast,\eta_\omega)$
	and $j = 1,\, \dots,\, \sd$, where $M_j$ does not depend on $\mu^{(1)}$. This in particular implies
	$\big| \big[ \partial  H^{{\mathcal V}_k,{\mathcal W}_k}[\mu]({\rm i} \omega) / \partial \mu_j \big]_{k\ell} \big| \leq M_j$
	for $k = 1,\, \dots,\, \sp$,\, $\ell = 1,\, \dots,\, \sm$. With $M := \max \{ M_j \; | \; j = 1,\, \dots,\, \sd \}$, 
	for every $\widetilde{\mu}, \mu \in \overline{{\mathcal B}}(\mu_\ast,\eta_\mu), \omega \in \overline{{\mathcal B}}(\omega_\ast,\eta_\omega)$,
	by the mean value theorem we obtain
	\begin{equation*}
	\begin{split}
	\big| \big[ H^{{\mathcal V}_k,{\mathcal W}_k}\big[\widetilde{\mu}\big]({\rm i} \omega) \big]_{k\ell}	-   
				\big[ H^{{\mathcal V}_k,{\mathcal W}_k}[\mu]({\rm i} \omega) \big]_{k\ell} \big|
			& \leq  
			\big|  \nabla_\mu [ H^{{\mathcal V}_k,{\mathcal W}_k}[\widehat{\mu}]({\rm i} \omega) ]_{k\ell}^\mathsf{T} \big(\widetilde{\mu} - \mu\big) \big|  \\
			& \leq \sum_{j=1}^{\sd} M_j \big| \widetilde{\mu}_j - \mu_j \big|	\leq	\sd M \big\| \widetilde{\mu} - \mu \big\|_2
	\end{split}
	\end{equation*}
	for some $\widehat{\mu} \in \overline{{\mathcal B}}(\mu_\ast,\eta_\mu)$, where 
	$\nabla_\mu \big[ H^{{\mathcal V}_k,{\mathcal W}_k}\big[\widehat{\mu}\big]\big({\rm i} \omega\big) \big]_{k\ell}$ denotes the gradient
	of $\mu \mapsto \big[ H^{{\mathcal V}_k,{\mathcal W}_k}[\mu]({\rm i} \omega) \big]_{k\ell}$ at $\widehat{\mu}$.
	It follows that
	\[
	       \big\| H^{{\mathcal V}_k,{\mathcal W}_k}\big[\widetilde{\mu}\big]({\rm i} \omega)	-  
			H^{{\mathcal V}_k,{\mathcal W}_k}[\mu]({\rm i} \omega) \big\|_2
			\leq
		\sqrt{\sp \sm} \sd M \big\| \widetilde{\mu} - \mu \big\|_2,	
	\]
	where the Lipschitz constant $\sqrt{\sp \sm} \sd M$ is independent of $\mu^{(1)}$ as desired.
	\item[\bf (ii)] A similar proof as in part \textbf{(i)} applies but now by differentiating the function
	$(\mu,\omega) \mapsto H^{\mathcal{V}_k, \mathcal{W}_k}[\mu]({\rm i} \omega)$ with respect to $\omega$ instead of $\mu_j$.  
	\item[\bf (iii)] By Weyl's theorem \cite[Theorem~4.3.1]{Horn1985} and part \textbf{(i)} we have
	\begin{equation*}
	\begin{split}
		\big| \sigma^{{\mathcal V}_k, {\mathcal W}_k} \big(\widetilde{\mu}, \omega\big) - \sigma^{{\mathcal V}_k, {\mathcal W}_k} (\mu, \omega) \big| 
		&	\leq
		\big\| H^{{\mathcal V}_k, {\mathcal W}_k}\big[\widetilde{\mu}\big](\ri\omega) - H^{{\mathcal V}_k, {\mathcal W}_k}[\mu](\ri\omega) \big\|_2
		\leq
		\gamma \big\| \widetilde{\mu} - \mu \big\|_2,	\\
		\big| \sigma_2^{{\mathcal V}_k, {\mathcal W}_k} \big(\widetilde{\mu}, \omega\big) - \sigma_2^{{\mathcal V}_k, {\mathcal W}_k} (\mu, \omega) \big| 
		&	\leq
		\big\| H^{{\mathcal V}_k, {\mathcal W}_k}\big[\widetilde{\mu}\big](\ri\omega) - H^{{\mathcal V}_k, {\mathcal W}_k}[\mu](\ri\omega) \big\|_2
		\leq
		\gamma \big\| \widetilde{\mu} - \mu \big\|_2
	\end{split}
	\end{equation*}
	for all $\widetilde{\mu},\, \mu \in \overline{{\mathcal B}}(\mu_\ast,\eta_\mu)$ and $\omega \in \overline{{\mathcal B}}(\omega_\ast,\eta_\omega)$,
	hence we get the result.
	\item[\bf (iv)] Weyl's theorem and part \textbf{(ii)} combined imply
	\begin{equation*}
	\begin{split}
		\big| \sigma^{{\mathcal V}_k, {\mathcal W}_k} \big(\mu, \widetilde{\omega}\big) - \sigma^{{\mathcal V}_k, {\mathcal W}_k} (\mu, \omega) \big| 
		&	\leq
		\big\| H^{{\mathcal V}_k, {\mathcal W}_k}[\mu]\big(\ri\widetilde{\omega}\big) - H^{{\mathcal V}_k, {\mathcal W}_k}[\mu](\ri\omega) \big\|_2
		\leq
		\gamma \big| \widetilde{\omega} - \omega \big|,	\\
		\big| \sigma_2^{{\mathcal V}_k, {\mathcal W}_k} \big(\mu, \widetilde{\omega}\big) - \sigma_2^{{\mathcal V}_k, {\mathcal W}_k} (\mu, \omega) \big| 
		&	\leq
		\big\| H^{{\mathcal V}_k, {\mathcal W}_k}[\mu]\big(\ri\widetilde{\omega}\big) - H^{{\mathcal V}_k, {\mathcal W}_k}[\mu](\ri\omega) \big\|_2
		\leq
		\gamma \big| \widetilde{\omega} - \omega \big|
	\end{split}
	\end{equation*}
	for all $\mu \in \overline{{\mathcal B}}(\mu_\ast,\eta_\mu)$ and $\widetilde{\omega},\,\omega \in \overline{{\mathcal B}}(\omega_\ast,\eta_\omega)$ as claimed.
\end{enumerate}
\end{proof}

The lemma below asserts uniform upper bounds on the derivatives of
the largest singular value function for the reduced problem provided $D^{V_k, W_k}(\mu_\ast,\omega_\ast)$ 
is away from singularity. Its proof is inspired by \cite[Proposition~2.9]{Kangal2015}, and given
in the appendix.
\begin{lemma}\label{thm:bound_hd}
Suppose that Assumption \ref{assume:smoothness} holds and that ${\| \mu^{(1)} - \mu_\ast \|}_2$ is small enough.
Additionally, assume that for some $\beta > 0$, the point $\mu^{(1)}$ is such that
\[
        \sigma_{\min} (D(\mu_*,\ri \omega_\ast)) \geq \beta
         \quad  \text{and}  \quad
	\sigma_{\min} \big(D^{V_k, W_k}\big(\mu_\ast, \ri\omega_\ast \big) \big) \geq \beta.
\] 
Then there exist a $U$ and constants $\eta_\mu,\, \eta_\omega > 0$ independent of $\mu^{(1)}$ such that
\begin{multline*}
         \big| \sigma^{{\mathcal V}_k, {\mathcal W}_k}_{\chi_1} (\mu, \omega) \big| \leq U,  \quad
         \big| \sigma^{{\mathcal V}_k, {\mathcal W}_k}_{\chi_1 \chi_2} (\mu, \omega) \big| \leq U,  \quad
         \big| \sigma^{{\mathcal V}_k, {\mathcal W}_k}_{\chi_1 \chi_2 \chi_3} (\mu, \omega) \big| \leq U
         \\ \forall \, \mu \in \overline{{\mathcal B}}(\mu_\ast,\eta_\mu), \,\forall\, \omega \in \overline{{\mathcal B}}(\omega_\ast,\eta_\omega)
\end{multline*}
for all $\chi_1,\, \chi_2,\, \chi_3 \in \{\omega\} \cup \{ \mu_j \; | \; j = 1,\, \dots,\, \sd \}$.
\end{lemma}

The next result draws two important conclusions. First, the maximizers 
of $\sigma(\mu, \cdot)$ and $\sigma^{{\mathcal V}_k,{\mathcal W}_k}(\mu, \cdot)$ can be expressed 
as smooth functions of $\mu$ in a neighborhood of $\mu_\ast$. Second, $\left\| H[\cdot] \right\|_{{\mathcal H}_\infty}$, as 
well as its reduced counter-parts generated by the algorithms are smooth locally around $\mu_\ast$.
\begin{proposition}\label{thm:localrep_smoothness}
Suppose that Assumption \ref{assume:smoothness} holds and that $\| \mu^{(1)} - \mu_\ast \|_2$ is small enough. 
Furthermore, assume for some $\delta < 0$ and $\beta > 0$ that the point $\mu^{(1)}$ is such that 
\begin{equation}\label{eq:secder_bounds}
       \sigma_{\omega \omega}(\mu_\ast,\omega_\ast) \leq \delta
       \quad \text{and}   \quad
	\sigma^{{\mathcal V}_k, {\mathcal W}_k}_{\omega \omega}(\mu_\ast,\omega_\ast)  \leq \delta, 
\end{equation}
as well as
\begin{equation*}
         \sigma_{\min} (D(\mu_*,\ri \omega_\ast)) \geq \beta
         \quad  \text{and}  \quad
	\sigma_{\min} \big(D^{V_k, W_k}\big(\mu_\ast, \ri\omega_\ast \big) \big) \geq \beta.
\end{equation*}
Then for some $\eta_{\mu,0},\, \eta_{\omega,0},\, \varepsilon > 0$ independent of $\mu^{(1)}$, 
the following assertions hold:
\begin{enumerate}
	\item[\bf (i)] There exists a unique continuous function 
	$\boldsymbol{\omega} : \overline{{\mathcal B}}(\mu_\ast,\eta_{\mu,0}) \rightarrow \overline{{\mathcal B}}(\omega_\ast.\eta_{\omega,0})$
	that is three times continuously differentiable in the interior of $\overline{{\mathcal B}}(\mu_\ast,\eta_{\mu,0})$ such that 
	\[
		\boldsymbol{\omega}(\mu_\ast) = \omega_\ast \quad \text{and}	\quad
				\sigma_\omega(\mu,\boldsymbol{\omega}(\mu))=0		
									 \quad	\forall \mu \in \overline{{\mathcal B}}(\mu_\ast, \eta_{\mu,0}).
	\]
	Furthermore,
	$\sigma_{\omega \omega}(\mu,\boldsymbol{\omega}(\mu)) \leq \delta/2$ for all $\mu \in \overline{{\mathcal B}}(\mu_\ast, \eta_{\mu,0})$.
	\item[\bf (ii)] There exists a unique continuous function
	$\boldsymbol{\omega}^{{\mathcal V}_k, {\mathcal W}_k} :$ $\overline{{\mathcal B}}(\mu_\ast,\eta_{\mu,0}) \rightarrow \overline{{\mathcal B}}(\omega_\ast,\eta_{\omega,0})$ that is three times continuously differentiable in the interior of $\overline{{\mathcal B}}(\mu_\ast,\eta_{\mu,0})$ such that 
	\[
		\boldsymbol{\omega}^{{\mathcal V}_k, {\mathcal W}_k}\big(\mu^{(k)}\big) = \omega^{(k)}
			\quad	\text{and}	\quad
		\sigma_\omega^{{\mathcal V}_k, {\mathcal W}_k}\big(\mu,\boldsymbol{\omega}^{{\mathcal V}_k,{\mathcal W}_k}(\mu)\big) = 0
		\quad 	\forall \mu \in \overline{{\mathcal B}}(\mu_\ast, \eta_{\mu,0}).
	\]
	Additionally, 
 $\sigma_{\omega \omega}^{{\mathcal V}_k, {\mathcal W}_k}\big(\mu,\boldsymbol{\omega}^{{\mathcal V}_k,{\mathcal W}_k}(\mu)\big) \leq \delta/2$
 	for all $\mu \in \overline{{\mathcal B}}(\mu_\ast, \eta_{\mu,0})$.
	\item[\bf (iii)] We have 
	\[
	      \sigma(\mu, \boldsymbol{\omega}(\mu)) - \sigma_2(\mu,\boldsymbol{\omega}(\mu))  \geq \varepsilon,
	\]
	and the unique global maximizer of $\sigma(\mu, \cdot)$ is given by $\boldsymbol{\omega}(\mu)$. In particular, for all $\mu \in \overline{{\mathcal B}}(\mu_\ast, \eta_{\mu,0})$ it holds that
	\[
			 \sigma(\mu,\boldsymbol{\omega}(\mu))	=	{\| H[\mu] \|}_{{\mathcal H}_\infty}.
	\]
	\item[\bf (iv)] We have
	\[
		\sigma^{{\mathcal V}_k, {\mathcal W}_k}\big(\mu, \boldsymbol{\omega}^{{\mathcal V}_k, {\mathcal W}_k}(\mu)\big) -  
				\sigma_2^{{\mathcal V}_k, {\mathcal W}_k}\big(\mu,\boldsymbol{\omega}^{{\mathcal V}_k,{\mathcal W}_k}(\mu)\big)
		\geq \varepsilon,
	\]
	and the unique global maximizer and stationary point of $\sigma^{{\mathcal V}_k, {\mathcal W}_k}(\mu, \cdot)$ in 
	$\overline{{\mathcal B}}(\omega_\ast,\eta_{\omega,0})$ is $\boldsymbol{\omega}^{{\mathcal V}_k, {\mathcal W}_k}(\mu)$ for all 
	$\mu \in \overline{{\mathcal B}}(\mu_\ast, \eta_{\mu,0})$.
\end{enumerate}
\end{proposition}
	
\begin{proof}
As argued in the opening of the proof of Lemma~\ref{thm:bound_hd}, we have
\begin{equation}\label{eq:sval_gap}
	\sigma(\mu, \omega) - \sigma_2(\mu, \omega)  \geq   \widehat{\varepsilon}
	\quad	\forall (\mu, \omega) \in 
			\overline{{\mathcal B}}(\mu_\ast, \widehat{\eta}_{\mu}) \times \overline{{\mathcal B}}(\omega_\ast, \widehat{\eta}_{\omega})
\end{equation}
for some $\widehat{\varepsilon} > 0,\, \widehat{\eta}_{\mu} > 0,\, \widehat{\eta}_{\omega} > 0$, and 
\begin{equation}\label{eq:sval_gap2}	
	\sigma^{{\mathcal V}_k, {\mathcal W}_k}(\mu, \omega) - \sigma^{{\mathcal V}_k, {\mathcal W}_k}_2(\mu, \omega)	
							\geq  \varepsilon
	\quad	\forall (\mu, \omega) \in 
			\overline{{\mathcal B}}(\mu_\ast, \widetilde{\eta}_{\mu}) \times \overline{{\mathcal B}}(\omega_\ast, \widetilde{\eta}_{\omega})
\end{equation}
for some $\varepsilon \in (0, \widehat{\varepsilon} ),\, \widetilde{\eta}_{\mu} > 0,\, \widetilde{\eta}_{\omega} > 0$. 
An important point here is that $\varepsilon,\, \widetilde{\eta}_{\mu},\, \widetilde{\eta}_{\omega}$ do not depend on $\mu^{(1)}$.
The function $\sigma(\cdot,\cdot)$ is real analytic in the interior of
$\overline{{\mathcal B}}(\mu_\ast, \widehat{\eta}_{\mu}) \times \overline{{\mathcal B}}(\omega_\ast, \widehat{\eta}_{\omega})$, whereas
$\sigma^{{\mathcal V}_k, {\mathcal W}_k}(\cdot,\cdot)$ is real analytic in the interior of
$\overline{{\mathcal B}}(\mu_\ast, \widetilde{\eta}_{\mu}) \times \overline{{\mathcal B}}(\omega_\ast, \widetilde{\eta}_{\omega})$. 
Moreover, by Lemma~\ref{thm:bound_hd}, there exists a $\widetilde{\delta} > 0$ such that
\begin{equation}\label{eq:boundedness_3der}
	\left| \sigma^{{\mathcal V}_k, {\mathcal W}_k}_{\omega \omega \omega}(\mu,\omega) \right|  \leq  \widetilde{\delta}
\end{equation}
holds uniformly for all $(\mu, \omega)$ in a neighborhood of $(\mu_\ast, \omega_\ast)$, where $\widetilde{\delta}$
and the neighborhood are independent of $\mu^{(1)}$.
Now we prove the four statements of the proposition:

\textbf{(i)} Since $\sigma(\cdot,\cdot)$ is real analytic with continuous second derivatives
in a neighborhood of $(\mu_\ast, \omega_\ast)$, its second derivative $\sigma_{\omega \omega}(\cdot,\cdot)$ 
must be bounded from above by $\delta/2$ in another neighborhood of $(\mu_\ast, \omega_\ast)$.
Then the assertion follows immediately from the implicit function theorem.

\textbf{(ii)} Due to \eqref{eq:boundedness_3der} 
the condition $\sigma^{{\mathcal V}_k,{\mathcal W}_k}_{\omega \omega} (\mu, \omega) \leq \delta /2$ 
must hold in an open neighborhood ${\mathcal N}$ of $(\mu_\ast, \omega_\ast)$ independent of $\mu^{(1)}$.
Additionally, observe  that $\omega^{(k)} \rightarrow \omega_\ast,\, \mu^{(k)} \rightarrow \mu_\ast$ 
as $\mu^{(1)} \rightarrow \mu_\ast$ due to 
\[
	\sigma(\mu_\ast, \omega_\ast)	=
	\left\| H[\mu_\ast] \right\|_{{\mathcal H}_\infty} = 
	    \lim_{\mu^{(1)} \rightarrow \mu_\ast} \big\| H\big[\mu^{(k)}\big] \big\|_{{\mathcal H}_\infty} = 
	         \lim_{\mu^{(1)} \rightarrow \mu_\ast} \sigma\big(\mu^{(k)}, \omega^{(k)}\big),
\]
as well as the uniqueness of $\omega_\ast$ as the maximizer of $\sigma(\mu_\ast, \cdot)$ and
the continuity of $\sigma(\cdot,\cdot)$. We assume ${\| \mu^{(1)} - \mu_\ast \|}_2$ is small enough so that
$\big(\mu^{(k)},\, \omega^{(k)}\big) \in {\mathcal N}$, particularly
$\sigma^{{\mathcal V}_k,{\mathcal W}_k}_{\omega \omega} \big(\mu^{(k)}, \omega^{(k)}\big) \leq \delta /2 < 0$.
Now the assertion again follows from the implicit function theorem. The uniformity of the radii
$\eta_{\mu,0},\,\eta_{\omega,0}$ over all $\mu^{(1)}$ (as long as ${\| \mu^{(1)} - \mu_\ast \|}_2$ is small enough) 
follows from the uniform upper bound $\delta/2$ on the second derivatives.

\textbf{(iii)} Assume $\eta_{\mu,0} \leq \widehat{\eta}_{\mu}$, $\eta_{\omega,0} \leq \widehat{\eta}_{\omega}$ without loss of generality,
where $\widehat{\eta}_{\mu},\, \widehat{\eta}_{\omega}$ are as in \eqref{eq:sval_gap}. But then for 
$\mu \in \overline{{\mathcal B}}(\mu_\ast, \eta_{\mu,0}) \subseteq \overline{{\mathcal B}}(\mu_\ast, \widehat{\eta}_{\mu})$, we have 
$\boldsymbol{\omega}(\mu) \in \overline{{\mathcal B}}(\omega_\ast, \eta_{\omega,0}) \subseteq \overline{{\mathcal B}}(\omega_\ast, \widehat{\eta}_{\omega})$.
Hence, \eqref{eq:sval_gap} implies 
$\sigma(\mu, \boldsymbol{\omega}(\mu)) - \sigma_2(\mu,\boldsymbol{\omega}(\mu)) \geq \widehat{\varepsilon} > \varepsilon$.

To show that $\boldsymbol{\omega}(\mu)$ is the unique global maximizer of $\sigma(\mu,\cdot)$ for all $\mu \in \overline{{\mathcal B}}(\mu_\ast, \eta_{\mu,0})$, we introduce
\begin{equation*}
	\delta_1(\mu) := \sup \{	\sigma(\mu, \omega) \; | \; \omega \in \overline{{\mathcal B}}(\omega_\ast, \eta_{\omega,0}) 	\}, \quad 
	\delta_2(\mu) := \sup \{	\sigma(\mu, \omega) \; | \; \omega \in {\mathbb R} \setminus \overline{{\mathcal B}}(\omega_\ast, \eta_{\omega,0}) 	\},
\end{equation*}
and let $\delta_\ast := \delta_1(\mu_\ast) - \delta_2(\mu_\ast) > 0$. 
As argued at the beginning of the proof of Lemma~\ref{thm:bound_hd}, there exists a neighborhood $\widetilde{\mathcal N}$
of $(\mu_\ast, \omega_\ast)$ where the transfer function $(\mu, \omega) \mapsto H[\mu]({\rm i} \omega)$ is continuously differentiable. 
As a result, the largest singular value function $\sigma(\cdot, \cdot)$ 
is Lipschitz continuous, say with the Lipschitz constant $\zeta$ over $\widetilde{\mathcal N}$
which we assume contains $\overline{{\mathcal B}}(\mu_\ast, \eta_{\mu,0}) \times \overline{{\mathcal B}}(\omega_\ast, \eta_{\omega,0})$ without
loss of generality. The functions $\delta_1(\cdot)$ and $\delta_2(\cdot)$ are also Lipschitz continuous with
the Lipschitz constant $\zeta$ over $\overline{{\mathcal B}}(\mu_\ast, \eta_{\mu,0})$, see \cite[Lemma~8 (ii)]{Mengi2018} (that concerns the minimization of a 
smallest singular value rather than the maximization of a largest singular value as in here, but the proof over there can be modified 
in a straightforward manner).  We furthermore assume $\eta_{\mu,0} < \delta_\ast / (4 \zeta)$ 
without loss of generality (since we can choose $\eta_{\mu,0}$ as small as we wish), so
\begin{equation*}
	\delta_1(\mu) \geq \delta_1(\mu_\ast)	- \delta_\ast/4	\quad	\text{and}		\quad
	\delta_2(\mu_\ast) \geq \delta_2(\mu) - \delta_\ast/4
\end{equation*}
for all $\mu \in \overline{{\mathcal B}}(\mu_\ast, \eta_{\mu,0})$ by the Lipschitz continuity of $\delta_1(\cdot)$ and $\delta_2(\cdot)$.
These inequalities combined with $\delta_1(\mu_\ast) - \delta_2(\mu_\ast) = \delta_\ast$ yield
\begin{equation*}
	\delta_1(\mu)	-	\delta_2(\mu)	\geq		\delta_1(\mu_\ast)	-	\delta_2(\mu_\ast)	-	\delta_\ast / 2
								=	\delta_\ast / 2
\end{equation*}
for all $\mu \in \overline{{\mathcal B}}(\mu_\ast, \eta_{\mu,0})$. This means that any global maximizer 
$\widetilde{\boldsymbol{\omega}}(\mu)$ of $\sigma(\mu, \cdot)$ lies in the interior of 
$\overline{{\mathcal B}}(\omega_\ast, \eta_{\omega,0})$. Since $\sigma(\cdot, \cdot)$ is differentiable in a neighborhood of
$\overline{{\mathcal B}}(\mu_*, \eta_{\mu,0}) \times \overline{{\mathcal B}}(\omega_\ast, \eta_{\omega,0})$, we must have
$\sigma_{\omega}(\mu, \widetilde{\boldsymbol{\omega}}(\mu)) = 0$. We conclude from part \textbf{(i)} that
$\widetilde{\boldsymbol{\omega}}(\mu) = \boldsymbol{\omega}(\mu)$ and that it is the unique global
maximizer of $\sigma(\mu,\cdot)$.

\textbf{(iv)} We assume without loss of generality that $\eta_{\mu,0} \leq \widetilde{\eta}_{\mu}$ and $\eta_{\omega,0} \leq \widetilde{\eta}_{\omega}$ .
Consequently, 
$\boldsymbol{\omega}^{{\mathcal V}_k, {\mathcal W}_k} (\mu) \in \overline{{\mathcal B}}(\omega_\ast, \eta_{\omega,0}) 
							\subseteq \overline{{\mathcal B}}(\omega_\ast, \widetilde{\eta}_{\omega})$
for all $\mu \in \overline{{\mathcal B}}(\mu_\ast, \eta_{\mu,0}) \subseteq \overline{{\mathcal B}}(\mu_\ast, \widetilde{\eta}_{\mu})$, so
\eqref{eq:sval_gap2} yields
$\sigma^{{\mathcal V}_k, {\mathcal W}_k}\big(\mu, \boldsymbol{\omega}^{{\mathcal V}_k, {\mathcal W}_k}(\mu)\big) 
	- \sigma^{{\mathcal V}_k, {\mathcal W}_k}_2\big(\mu,\boldsymbol{\omega}^{{\mathcal V}_k, {\mathcal W}_k}(\mu)\big) \geq \varepsilon$
for such $\mu$. The uniqueness of $\boldsymbol{\omega}^{{\mathcal V}_k, {\mathcal W}_k} (\mu)$ as the stationary point 
of $\sigma^{{\mathcal V}_k, {\mathcal W}_k}(\mu, \cdot)$ in $\overline{{\mathcal B}}(\omega_\ast, \eta_{\omega,0})$ is immediate from the implicit 
function theorem. Additionally, without loss of generality, we can assume 
$\overline{{\mathcal B}}(\mu_\ast, \eta_{\mu,0}) \times \overline{{\mathcal B}}(\omega_\ast, \eta_{\omega,0}) \subseteq {\mathcal N}$
where ${\mathcal N}$ is the neighborhood of $(\mu_\ast, \omega_\ast)$ as in part \textbf{(ii)} over
which $\sigma^{{\mathcal V}_k,{\mathcal W}_k}_{\omega \omega} (\mu, \omega)$ $\leq \delta /2 < 0$.
This means $\sigma^{{\mathcal V}_k, {\mathcal W}_k}(\mu, \cdot)$ is strictly concave in $\overline{{\mathcal B}}(\omega_\ast, \eta_{\omega,0})$. 
Thus, the unique stationary point $\boldsymbol{\omega}^{{\mathcal V}_k, {\mathcal W}_k} (\mu)$ must also be the unique global maximizer
of $\sigma^{{\mathcal V}_k, {\mathcal W}_k}(\mu, \cdot)$ in $\overline{{\mathcal B}}(\omega_\ast, \eta_{\omega,0})$.
\end{proof}

\begin{remark}
	The second condition in (\ref{eq:secder_bounds}) can be dropped in theory by including
	additional vectors in the subspaces and doubling the subspace dimensions as follows: The interpolation
	property	
	\[
		\sigma_\omega\big( \mu^{(k)}, \omega^{(k-1)}\big) =	
				\sigma^{{\mathcal V}_k, {\mathcal W}_k}_\omega\big( \mu^{(k)}, \omega^{(k-1)}\big)
	\]	
	can be achieved, for instance, by the inclusions
	\begin{align*}
	 \operatorname{Col}\left(\left(\ri\omega^{(k-1)}E\big(\mu^{(k)}\big) - A\big(\mu^{(k)}\big)\right)^{-1}B\big(\mu^{(k)}\big)\right) &\subseteq {\mathcal V}_k, \\ 
	 \operatorname{Col}\left(\left(\ri\omega^{(k-1)}E\big(\mu^{(k)}\big) - A\big(\mu^{(k)}\big)\right)^{-\ast}C\big(\mu^{(k)}\big)^\ast \right) 
		 &\subseteq {\mathcal W}_k,
	\end{align*}
 	when $\sm = \sp$. By the mean value theorem, this would lead to
	\begin{align*}
		\sigma_{\omega\omega}\big(\mu^{(k)}, \xi^{(k)}\big)
					& =
		\frac{  \sigma_\omega\big( \mu^{(k)}, \omega^{(k)}\big)  - \sigma_\omega\big( \mu^{(k)}, \omega^{(k-1)}\big)  }
		{\omega^{(k)} - \omega^{(k-1)}}		\\
					& =
		\frac{  \sigma^{{\mathcal V}_k, {\mathcal W}_k}_\omega\big( \mu^{(k)}, \omega^{(k)}\big)  - 
							\sigma^{{\mathcal V}_k, {\mathcal W}_k}_\omega\big( \mu^{(k)}, \omega^{(k-1)}\big) }
		{\omega^{(k)} - \omega^{(k-1)}}
					=
		\sigma^{{\mathcal V}_k, {\mathcal W}_k}_{\omega\omega}\big(\mu^{(k)}, \widetilde{\xi}^{(k)}\big)
	\end{align*}
	for some $ \xi^{(k)}, \, \widetilde{\xi}^{(k)}$ in the open interval with the end-points $\omega^{(k-1)}$,
	$\omega^{(k)}$ so that
	\[
			\left|
			\sigma_{\omega \omega}\big( \mu^{(k)}, \omega^{(k)}\big)
					-
			\sigma^{{\mathcal V}_k, {\mathcal W}_k}_{\omega \omega}\big( \mu^{(k)}, \omega^{(k)}\big)
			\right| = 
			\mathcal{O}\left(\big| \omega^{(k)} - \omega^{(k-1)} \big|\right).
	\]
	Hence, by the continuity of the second derivatives with respect to $\omega$ and 
	$\big(\mu^{(k)}, \omega^{(k)}\big) \rightarrow (\mu_\ast, \omega_\ast)$,
	the condition $\sigma_{\omega \omega}(\mu_\ast,\omega_\ast) =: \widehat{\delta} < 0$ would imply 
	$\sigma^{{\mathcal V}_k, {\mathcal W}_k}_{\omega \omega}(\mu_\ast,\omega_\ast) < \widehat{\delta}/2 < 0$
	provided $\mu^{(1)}$ is chosen sufficiently close to $\mu_\ast$. 
\end{remark}
	
The main conclusion of the next result is that the higher-order derivatives of 
$\mu \mapsto \sigma^{{\mathcal V}_k, {\mathcal W}_k}\big(\mu, \boldsymbol{\omega}^{{\mathcal V}_k, {\mathcal W}_k}(\mu)\big)$ 
 are uniformly bounded in absolute value by a constant.
\begin{lemma}[Uniform boundedness of higher-order derivatives]\label{thm:3der_bounded}
Suppose that the assumptions of Proposition~\ref{thm:localrep_smoothness} hold. Furthermore, let $\eta_{\mu,0}$ and $\boldsymbol{\omega}^{{\mathcal V}_k, {\mathcal W}_k}(\mu)$ be as in Proposition~\ref{thm:localrep_smoothness}, and let 
\begin{equation}\label{eq:reduced_sfun}
	\widetilde{\sigma}^{{\mathcal V}_k, {\mathcal W}_k}(\mu) :=  \sigma^{{\mathcal V}_k, {\mathcal W}_k}(\mu, \boldsymbol{\omega}^{{\mathcal V}_k, {\mathcal W}_k}(\mu)).
\end{equation}
Then for every $\widehat{\eta}_{\mu,0} \in (0, \eta_{\mu,0})$ there exists $\gamma > 0$ independent of $\mu^{(1)}$ such that for all 
$\mu \in \overline{{\mathcal B}}(\mu_\ast,\widehat{\eta}_{\mu,0})$, 
we have \vspace*{1ex}
\begin{enumerate}
\item[\bf (i)]
$
	\displaystyle
	\left|
	\frac{\partial^2 {\| H[\mu] \|}_{{\mathcal H}_\infty}}{\partial \mu_q \partial \mu_r }
	\right| \leq \gamma
		\quad	\text{and} \quad
	\left|
	\widetilde{\sigma}^{{\mathcal V}_k, {\mathcal W}_k}_{\mu_q  \mu_r} (\mu)
	\right| \leq \gamma, \quad 
		 q,\, r = 1,\,\dots,\, \sd,
$ \vspace*{1ex}
\item[\bf (ii)]
$
	\displaystyle
	\left|
	\frac{\partial^3 {\| H[\mu] \|}_{{\mathcal H}_\infty}}{\partial \mu_q \partial \mu_r \partial \mu_\ell}
	\right|
		\leq  \gamma
		\quad \text{and} \quad
	 \left|
		 \widetilde{\sigma}^{{\mathcal V}_k, {\mathcal W}_k}_{\mu_q  \mu_r  \mu_\ell}(\mu)
	\right|
		\leq \gamma,
		\quad
		 q,\, r,\, \ell = 1,\,\dots,\, \sd.
$
\end{enumerate}
\end{lemma}
\begin{proof}
\textbf{(i)} By Proposition~\ref{thm:localrep_smoothness}, the functions ${\| H[\cdot] \|}_{{\mathcal H}_\infty}$ and 
$\widetilde{\sigma}^{{\mathcal V}_k, {\mathcal W}_k}(\cdot)$ are three times continuously differentiable in a neighborhood 
of $\overline{{\mathcal B}}(\mu_\ast, \widehat{\eta}_{\mu,0})$. The first assertion, that is the boundedness 
of the second derivatives of ${\| H[\cdot] \|}_{{\mathcal H}_\infty}$ in $\overline{{\mathcal B}}(\mu_\ast, \widehat{\eta}_{\mu,0})$, 
is immediate. Let us prove the existence of a uniform $\gamma > 0$ such that
\begin{equation}\label{eq:bound_2der_rproblems}
	 \left| \widetilde{\sigma}^{{\mathcal V}_k, {\mathcal W}_k}_{\mu_q  \mu_r} (\mu) \right| \leq \gamma 
	\quad \forall\, \mu \in \overline{{\mathcal B}}(\mu_\ast,\widehat{\eta}_{\mu,0}) 
\end{equation}
for $q,\, r = 1,\,\dots,\, d$ independent of $\mu^{(1)}$. To this end, we first observe
\begin{equation}\label{eq:reduced_sval_2der}
	 \widetilde{\sigma}^{{\mathcal V}_k, {\mathcal W}_k}_{\mu_q  \mu_r} (\mu)
	 			 = 
	  \sigma^{{\mathcal V}_k, {\mathcal W}_k}_{\mu_q  \mu_r} \big(\mu, \boldsymbol{\omega}^{{\mathcal V}_k, {\mathcal W}_k}(\mu)\big) 
				  + 	
	  	\sigma^{{\mathcal V}_k, {\mathcal W}_k}_{\mu_q \omega} \big(\mu, \boldsymbol{\omega}^{{\mathcal V}_k, {\mathcal W}_k}(\mu)\big)
		\boldsymbol{\omega}^{{\mathcal V}_k, {\mathcal W}_k}_{\mu_r}(\mu).
\end{equation}
The function $\boldsymbol{\omega}^{{\mathcal V}_k, {\mathcal W}_k}(\cdot)$ is implicitly defined by the equation
$
	\sigma^{{\mathcal V}_k, {\mathcal W}_k}_\omega\big(\mu, \boldsymbol{\omega}^{{\mathcal V}_k, {\mathcal W}_k}(\mu)\big)  = 0
$
for $\mu$ near $\mu^{(k)}$. Differentiating this equation with respect to $\mu_r$ yields
\[
	\boldsymbol{\omega}^{{\mathcal V}_k, {\mathcal W}_k}_{\mu_r}(\mu) 
		 = 
	-   \frac{ \sigma^{{\mathcal V}_k, {\mathcal W}_k}_{\mu_r \omega} \big(\mu,\boldsymbol{\omega}^{{\mathcal V}_k, {\mathcal W}_k}(\mu)\big)} {\sigma^{{\mathcal V}_k, {\mathcal W}_k}_{\omega\omega} \big(\mu,\boldsymbol{\omega}^{{\mathcal V}_k, {\mathcal W}_k}(\mu)\big)},
\]
which we plug into \eqref{eq:reduced_sval_2der} to obtain
\begin{equation}\label{eq:reduced_sval_2derf}
	\widetilde{\sigma}^{{\mathcal V}_k, {\mathcal W}_k}_{\mu_q  \mu_r} (\mu)
	 			  =  
\sigma^{{\mathcal V}_k, {\mathcal W}_k}_{\mu_q  \mu_r} \big(\mu,\boldsymbol{\omega}^{{\mathcal V}_k, {\mathcal W}_k}(\mu)\big) 
				 -	
		\frac{  \sigma^{{\mathcal V}_k, {\mathcal W}_k}_{\mu_q \omega} \big(\mu,\boldsymbol{\omega}^{{\mathcal V}_k, {\mathcal W}_k}(\mu)\big)   
					\sigma^{{\mathcal V}_k, {\mathcal W}_k}_{\mu_r \omega} \big(\mu,\boldsymbol{\omega}^{{\mathcal V}_k, {\mathcal W}_k}(\mu)\big)  }
		{\sigma^{{\mathcal V}_k, {\mathcal W}_k}_{\omega\omega} \big(\mu,\boldsymbol{\omega}^{{\mathcal V}_k, {\mathcal W}_k}(\mu)\big) }.
\end{equation}
By part \textbf{(ii)} of Proposition~\ref{thm:localrep_smoothness}, we have 
$\sigma_{\omega \omega}^{{\mathcal V}_k, {\mathcal W}_k}\big(\mu,\boldsymbol{\omega}^{{\mathcal V}_k,{\mathcal W}_k}(\mu)\big) \leq \delta/2 < 0$
for all $\mu \in \overline{{\mathcal B}}(\mu_\ast, \eta_{\mu,0})$ independent of $\mu^{(1)}$. 
Additionally, by Lemma~\ref{thm:bound_hd}, all mixed second derivatives of $\sigma^{{\mathcal V}_k, {\mathcal W}_k}(\cdot,\cdot)$ are bounded
from above in absolute value uniformly in 
$\overline{{\mathcal B}}(\mu_\ast,\eta_\mu) \times \overline{{\mathcal B}}(\omega_\ast,\eta_\omega) \supseteq 
			\overline{{\mathcal B}}(\mu_\ast,\eta_{\mu,0}) \times \overline{{\mathcal B}}(\omega_\ast,\eta_{\omega,0})$
(to be precise we assume the inclusion without loss of generality as we can choose $\eta_{\mu,0},\, \eta_{\omega,0}$ as small as we wish), 
where the upper bound is independent of $\mu^{(1)}$. Hence, we conclude with (\ref{eq:bound_2der_rproblems}) as desired.

\textbf{(ii)} The boundedness of the third derivatives of ${\| H[\cdot] \|}_{{\mathcal H}_\infty}$ in $\overline{{\mathcal B}}(\mu_\ast,\eta_{\mu,0})$ 
is immediate from three times continuous differentiability of ${\| H[\cdot] \|}_{{\mathcal H}_\infty}$ in a neighborhood of $\overline{{\mathcal B}}(\mu_\ast,\eta_{\mu,0})$.

The boundedness of the absolute values of the third derivatives of $\widetilde{\sigma}^{{\mathcal V}_k, {\mathcal W}_k}(\cdot)$
uniformly by a constant independent of $\mu^{(1)}$ can be established in a similar way as in part \textbf{(i)}. 
Specifically, by differentiating \eqref{eq:reduced_sval_2derf} with respect to $\mu_\ell$,  
it can be seen that $\widetilde{\sigma}^{{\mathcal V}_k, {\mathcal W}_k}_{\mu_q  \mu_r \mu_\ell} (\mu)$ is 
a ratio, where the expression in the numerator is a sum of products of the mixed second derivatives 
$\sigma^{{\mathcal V}_k, {\mathcal W}_k}_{\chi_1, \chi_2} \big(\mu,\boldsymbol{\omega}^{{\mathcal V}_k, {\mathcal W}_k}(\mu)\big)$
and third derivatives $\sigma^{{\mathcal V}_k, {\mathcal W}_k}_{\chi_1, \chi_2 \chi_3} \big(\mu,\boldsymbol{\omega}^{{\mathcal V}_k, {\mathcal W}_k}(\mu)\big)$
for $\chi_1, \chi_2, \chi_3 \in \{\omega, \mu_q , \mu_r, \mu_\ell \}$, while the expression in the denominator is 
$\sigma^{{\mathcal V}_k, {\mathcal W}_k}_{\omega\omega} \big(\mu,\boldsymbol{\omega}^{{\mathcal V}_k, {\mathcal W}_k}(\mu)\big)^3$.
Hence, once again, the conclusion
\[
 \left| \widetilde{\sigma}^{{\mathcal V}_k, {\mathcal W}_k}_{\mu_q  \mu_r \mu_\ell} (\mu) \right| \leq \gamma 
	\quad \forall\, \mu \in \overline{{\mathcal B}}(\mu_\ast,\widehat{\eta}_{\mu,0}) 
\]
for $q,\, r,\, \ell = 1,\,\dots,\, \sd$ for some $\gamma$ independent of $\mu^{(1)}$ 
can be drawn from part \textbf{(ii)} of Proposition~\ref{thm:localrep_smoothness} and Lemma~\ref{thm:bound_hd}.
\end{proof}

By Lemma~\ref{thm:basic_int} and Lemma~\ref{thm:extended_int}, the reduced function 
$\widetilde{\sigma}^{{\mathcal V}_k, {\mathcal W}_k}(\cdot)$
Hermite interpolates the original $\Hinf$ function ${\| H[\cdot] \|}_{{\mathcal H}_\infty}$ at $\mu = \mu^{(k)},\, \mu^{(k-1)}$. 
Indeed, for the extended algorithm (Algorithm~\ref{alg2}), these Hermite interpolation properties also hold 
at $\mu = \mu^{(k,rq)}$ for each $r = 1,\,\dots,\, \sd, \, q = r,\, \dots,\, \sd$ by Lemma~\ref{thm:extended_int}.
From these observations, by also employing Lemma~\ref{thm:3der_bounded}, it is possible to conclude with 
an upper bound on the gap between the second derivatives of $\widetilde{\sigma}^{{\mathcal V}_k, {\mathcal W}_k}(\cdot)$ 
and $\| H[\cdot] \|_{{\mathcal H}_\infty}$ near $\mu_*$, which we formally state and prove next.
\begin{lemma}[Proximity of the second derivatives]\label{thm:acc_2der}
Suppose that the assumptions of Proposition~\ref{thm:localrep_smoothness} hold. Additionally, 
assume that $\nabla^2 {\| H[\mu_\ast] \|}_{{\mathcal H}_\infty}$ is invertible. 
Furthermore, let $\boldsymbol{\omega}^{{\mathcal V}_k, {\mathcal W}_k}(\cdot)$ be as 
in Proposition~\ref{thm:localrep_smoothness}, and 
$\widetilde{\sigma}^{{\mathcal V}_k, {\mathcal W}_k}(\cdot)$ be defined as in \eqref{eq:reduced_sfun}. 
Then there exists a $\zeta > 0$ such that the following statements hold for Algorithm~\ref{alg1} when $\sd =1$ 
and for Algorithm~\ref{alg2} independent of $\mu^{(1)}$: \vspace*{1ex}
\begin{enumerate}
	\item[\bf (i)]	We have $ \left\| \nabla^2 \big\| H\big[\mu^{(k)}\big] \big\|_{{\mathcal H}_\infty} 
						 -  
					\nabla^2 \widetilde{\sigma}^{{\mathcal V}_k, {\mathcal W}_k}\big(\mu^{(k)}\big) \right\|_2 
												\leq \zeta \big\| \mu^{(k)} - \mu^{(k-1)} \big\|_2$.\vspace*{1ex}
	\item[\bf (ii)] Both $\nabla^2 \big\| H\big[\mu^{(k)}\big] \big\|_{{\mathcal H}_\infty}$ and 
						$\nabla^2 \widetilde{\sigma}^{{\mathcal V}_k, {\mathcal W}_k}\big(\mu^{(k)}\big)$ are 
				invertible. \vspace*{1ex}
	\item[\bf (iii)] We have $ \left\| \left[ \nabla^2 \big\| H\big[\mu^{(k)}\big] \big\|_{{\mathcal H}_\infty} \right]^{-1} 
						-  
				\left[ \nabla^2 \widetilde{\sigma}^{{\mathcal V}_k, {\mathcal W}_k}\big(\mu^{(k)}\big) \right]^{-1} \right\|_2 
						\leq  \zeta \big\| \mu^{(k)} - \mu^{(k-1)} \big\|_2$.
\end{enumerate}
\end{lemma}
\begin{proof}
\textbf{(i)} We focus on Algorithm~\ref{alg2} only (the proof for Algorithm~\ref{alg1} with $\sd = 1$ proceeds similarly by defining $h^{(k)} := \mu^{(k-1)} - \mu^{(k)}$).
By parts \textbf{(iii)} and \textbf{(iv)} of Proposition~\ref{thm:localrep_smoothness}, the functions ${\| H[\cdot] \|}_{{\mathcal H}_\infty}$ and $\widetilde{\sigma}^{{\mathcal V}_k, {\mathcal W}_k}(\cdot)$ are three times 
differentiable in the interior of $\overline{{\mathcal B}}(\mu_\ast,\eta_{\mu,0})$ independent of $\mu^{(1)}$. 
Now choose $\mu^{(1)}$ close enough to $\mu_\ast$ so that 
$\overline{{\mathcal B}}\big(\mu^{(k)}, h^{(k)}\big) \subset \overline{{\mathcal B}}(\mu_\ast,\eta_{\mu,0})$,
as well as $\omega^{(k)}, \omega^{(k,rq)}$ belong to the interior of $\overline{{\mathcal B}}(\omega_\ast,\eta_{\omega,0})$ for $r = 1, \,\dots,\, \sd$ and
$q = r,\, \dots,\, \sd$ (observe that $\omega^{(k,rq)} \rightarrow \omega_\ast$ as $\mu^{(1)} \rightarrow \mu_\ast$ 
based on arguments similar to the ones for $\omega^{(k)} \rightarrow \omega_\ast$ as $\mu^{(1)} \rightarrow \mu_\ast$ 
given in the proof of part \textbf{(ii)} of Proposition~\ref{thm:localrep_smoothness}).

It follows that the functions
\begin{align*}
	\ell : [0,1] \rightarrow {\mathbb R}, \quad &
	\ell(\alpha) :=  \big\| H\big(\mu^{(k)} + \alpha h^{(k)} e_{rq}\big) \big\|_{{\mathcal H}_\infty}, \\
	\widetilde{\ell} : [0,1] \rightarrow {\mathbb R}, \quad &
	\widetilde{\ell}(\alpha) := \widetilde{\sigma}^{{\mathcal V}_k, {\mathcal W}_k}\big(\mu^{(k)} + \alpha h^{(k)} e_{rq}\big)
\end{align*}
are continuous and three times differentiable in $(0,1)$. Additionally, Lemma~\ref{thm:extended_int} implies that the following interpolation properties between these functions
\begin{equation}\label{eq:interpol_prop}
	\ell(0) = \widetilde{\ell}(0), \quad \ell'(0) = \widetilde{\ell}'(0) \quad \text{and} \quad \ell(1) = \widetilde{\ell}(1)
\end{equation}
are satisfied. 
To see the last equality at $\alpha = 1$, we observe
$$
	0  =  \sigma_\omega \big(\mu^{(k,rq)}, \omega^{(k,rq)}\big) =  \sigma^{{\mathcal V}_k, {\mathcal W}_k}_\omega \big(\mu^{(k,rq)}, \omega^{(k,rq)}\big)
$$
by Lemma~\ref{thm:extended_int}, so part \textbf{(iv)} of Proposition~\ref{thm:localrep_smoothness}, in particular the uniqueness 
of the stationary point $\boldsymbol{\omega}^{{\mathcal V}_k, {\mathcal W}_k}(\mu)$ of
$\sigma^{{\mathcal V}_k, {\mathcal W}_k}(\mu, \cdot)$ for all $\mu \in \overline{{\mathcal B}}(\mu_\ast,\eta_{\mu,0})$, implies
$\boldsymbol{\omega}^{{\mathcal V}_k, {\mathcal W}_k}\big(\mu^{(k,rq)}\big) = \omega^{(k,rq)}$ (as  
$\mu^{(k,rq)} \in \overline{{\mathcal B}}(\mu_\ast,\eta_{\mu,0})$ and $\omega^{(k,rq)} \in \overline{{\mathcal B}}(\omega_\ast,\eta_{\omega,0})$).
Hence, again by Lemma~\ref{thm:extended_int}, we have
\begin{align*}
	\ell(1)	= 
	\big\| H\big[\mu^{(k,pq)}\big] \big\|_{{\mathcal H}_\infty} & = \sigma^{{\mathcal V}_k, {\mathcal W}_k}\big(\mu^{(k,rq)}, \omega^{(k,rq)}\big) \\
		& = \sigma^{{\mathcal V}_k, {\mathcal W}_k}\big(\mu^{(k,rq)}, \boldsymbol{\omega}^{{\mathcal V}_k, {\mathcal W}_k}\big(\mu^{(k,rq)}\big)\big) = \widetilde{\ell}(1).
\end{align*}
By employing the interpolation properties in \eqref{eq:interpol_prop} in the Taylor expansions
\begin{align*}
	\ell(1) &= \ell(0) + \ell'(0) + \frac{1}{2} \ell''(0) + \frac{1}{6} \ell'''(\varepsilon), \\
	\widetilde{\ell}(1)	&= \widetilde{\ell}(0) + \widetilde{\ell}'(0) + \frac{1}{2} \widetilde{\ell}''(0) + \frac{1}{6} \widetilde{\ell}'''(\widetilde{\varepsilon})	
\end{align*}
for some $\varepsilon, \widetilde{\varepsilon} \in (0,1)$, we obtain
\begin{equation}\label{eq:linefuns_2der}
\begin{split}
	 \big[ h^{(k)} \big]^2 
	 e_{rq}^\mathsf{T} \left[ \nabla^2 \big\| H\big[\mu^{(k)}\big] \big\|_{{\mathcal H}_\infty}  - 
	 				\nabla^2  \widetilde{\sigma}^{{\mathcal V}_k, {\mathcal W}_k}\big(\mu^{(k)}\big) \right] e_{rq} \hskip 25ex \\
	= \ell''(0) - \widetilde{\ell}''(0) = 
				\frac{1}{3} \left( \widetilde{\ell}'''(\widetilde{\varepsilon}) - \ell'''(\varepsilon) \right) =  \mathcal{O}\left(\big[ h^{(k)} \big]^3\right),
\end{split}				
\end{equation}
where the constant hidden in the Landau symbol $\mathcal{O}$ is independent of $\mu^{(1)}$ due to Lemma~\ref{thm:3der_bounded}.
By considering particular values of $r = 1,\,\dots,\, \sd$ and $q = r,\,\dots,\, \sd$ in \eqref{eq:linefuns_2der},
we deduce
\begin{equation*} 
	\left|
	\frac{ \partial^2  \big\| H\big[\mu^{(k)}\big] \big\|_{{\mathcal H}_\infty} }{ \partial \mu_r \partial \mu_q }
				-
	\frac{\partial^2    \widetilde{\sigma}^{{\mathcal V}_k, {\mathcal W}_k}\big(\mu^{(k)}\big)} { \partial \mu_r \partial \mu_q }
	\right|
				= 
			\mathcal{O}\big(h^{(k)}\big).
\end{equation*}
Once again, the constant hidden in the Landau symbol $\mathcal{O}$ does not depend on $\mu^{(1)}$ in the latter equation.

\textbf{(ii)} By the continuity of $\nabla^2 {\| H[\cdot] \|}_{{\mathcal H}_\infty}$ in the interior of $\overline{{\mathcal B}}(\mu_\ast,\eta_{\mu,0})$, coupled 
with the assumption $\mu^{(k)} \rightarrow \mu_\ast$ as $\mu^{(1)} \rightarrow \mu_\ast$, we have 
$\lim_{\mu^{(1)} \rightarrow \mu_\ast} \nabla^2  \big\| H\big[\mu^{(k)}\big] \big\|_{{\mathcal H}_\infty}  = \nabla^2 {\| H[\mu_\ast] \|}_{{\mathcal H}_\infty}$. 
Consequently, $\nabla^2  \big\| H\big[\mu^{(k)}\big] \big\|_{{\mathcal H}_\infty}$ is invertible provided $\mu^{(1)}$ is sufficiently close to $\mu_\ast$.
In addition, from part~\textbf{(i)} we get
\[
	\nabla^2 {\| H[\mu_\ast] \|}_{{\mathcal H}_\infty}	=
	\lim_{\mu^{(1)} \rightarrow \mu_\ast } \nabla^2  \big\| H\big[\mu^{(k)}\big] \big\|_{{\mathcal H}_\infty}	=	
		\lim_{\mu^{(1)} \rightarrow \mu_\ast}  \nabla^2  \widetilde{\sigma}^{{\mathcal V}_k, {\mathcal W}_k}\big(\mu^{(k)}\big),
\]
implying also the invertibility of $\nabla^2  \widetilde{\sigma}^{{\mathcal V}_k, {\mathcal W}_k}\big(\mu^{(k)}\big)$ for $\mu^{(1)}$ close to $\mu_\ast$.

\textbf{(iii)} This statement follows from part \textbf{(i)} by employing the adjugate formulas for the inverses of 
$\nabla^2 \big\| H\big[\mu^{(k)}\big] \big\|_{{\mathcal H}_\infty}$ as well as 
$\nabla^2  \widetilde{\sigma}^{{\mathcal V}_k, {\mathcal W}_k}\big(\mu^{(k)}\big)$. For details, 
we refer to \cite[Lemma~2.8, part (ii)]{Kangal2015}.
\end{proof}

Now we are ready for the main rate-of-convergence result.
\begin{theorem}[Local superlinear convergence]\label{thm:super_converge}
Suppose that the assumptions of Proposition~\ref{thm:localrep_smoothness} hold. 
In particular, let $\boldsymbol{\omega}^{{\mathcal V}_k, {\mathcal W}_k}(\cdot)$ be as in 
Proposition~\ref{thm:localrep_smoothness}.  
Additionally, assume that the matrix $\nabla^2 {\| H[\mu_\ast] \|}_{{\mathcal H}_\infty}$ is invertible,
the point $\mu_\ast$ is strictly in the interior of $\underline{\Omega}$, 
and that the function 
$\sigma^{{\mathcal V}_k, {\mathcal W}_k}(\mu^{(k+1)}, \cdot)$ has a unique global maximizer,
say at $\widetilde{\omega}^{(k+1)}$,
\noindent
with $\lim_{\mu^{(1)} \rightarrow \mu_\ast} \widetilde{\omega}^{(k+1)} = \omega_\ast$. 
Regarding Algorithm~\ref{alg1} when $\sd = 1$ and Algorithm~\ref{alg2}, the following statement holds:
There exists a $C > 0$ independent of $\mu^{(1)}$ such that
\begin{equation}\label{eq:super_decay}
	\frac{ \big\| \mu^{(k+1)}  -  \mu_\ast \big\|_2}{ \big\|\mu^{(k)} - \mu_\ast \big\|_2 \max \big\{ \big\| \mu^{(k)} - \mu_\ast \big\|_2, \big\| \mu^{(k-1)} - \mu_\ast \big\|_2 \big\}} \leq  C.
\end{equation}
\end{theorem}
\begin{proof}
By Proposition~\ref{thm:localrep_smoothness}, both ${\| H[\cdot] \|}_{{\mathcal H}_\infty}$ and 
$\widetilde{\sigma}^{{\mathcal V}_k, {\mathcal W}_k}(\cdot)$ defined by \eqref{eq:reduced_sfun}
are twice Lipschitz continuously differentiable in the interior of the ball $\overline{{\mathcal B}}(\mu_\ast, \eta_{\mu,0})$. 
Now suppose $\mu^{(1)}$ is close enough to $\mu_\ast $ so that 
$\mu^{(k+1)},\, \mu^{(k)},\, \mu^{(k-1)}$ lie in the interior of $\overline{{\mathcal B}}(\mu_\ast, \eta_{\mu,0})$, whereas
$\widetilde{\omega}^{(k+1)}$ belongs to the interior of $\overline{{\mathcal B}}(\omega_\ast, \eta_{\omega,0})$, and
\begin{enumerate}
  \item[\bf (1)] $\overline{{\mathcal B}}\big(\mu^{(k)}, h^{(k)}\big) \subset \overline{{\mathcal B}}(\mu_\ast,\eta_{\mu,0})$ where $h^{(k)} := \big\| \mu^{(k)} - \mu^{(k-1)} \big\|_2$
  (recall that $\mu^{(k)} \rightarrow \mu_\ast$ and $h^{(k)} \rightarrow 0$ as $\mu^{(1)} \rightarrow \mu_\ast$),
  \item[\bf (2)] $\nabla^2 \big\| H\big[\mu^{(k)}\big] \big\|_{{\mathcal H}_\infty}$ and $\nabla^2  \widetilde{\sigma}^{{\mathcal V}_k, {\mathcal W}_k}\big(\mu^{(k)}\big)$ are invertible
  (part \textbf{(ii)} of Lemma~\ref{thm:acc_2der} ensures this as $\mu^{(1)}$ is chosen close to $\mu_\ast$).
\end{enumerate}

\medskip

By an application of Taylor's theorem with integral remainder we obtain 
\begin{equation*}
	0 = \nabla {\| H[\mu_\ast] \|}_{{\mathcal H}_\infty} = 
					\nabla \big\| H\big[\mu^{(k)}\big] \big\|_{{\mathcal H}_\infty}  + \int_0^1	\nabla^2 \big\| H\big[\mu^{(k)} + t \big(\mu_\ast - \mu^{(k)}\big)\big] \big\|_{{\mathcal H}_\infty}  \big(\mu_\ast - \mu^{(k)}\big) \mathrm{d}t,
\end{equation*}
which implies
\begin{equation}\label{eq:Taylor_beg}
\begin{split}
 0	= \left[ \nabla^2 \big\| H\big[\mu^{(k)}\big] \big\|_{{\mathcal H}_\infty} \right]^{-1} \nabla \big\| H\big[\mu^{(k)}\big] \big\|_{{\mathcal H}_\infty}   + 	 \big(\mu_\ast - \mu^{(k)} \big)  +  
				  \left[ \nabla^2 \big\| H\big[\mu^{(k)}\big] \big\|_{{\mathcal H}_\infty} \right]^{-1}
							\\ 
	\times \int_0^1	\left[ \nabla^2 \big\| H\big[\mu^{(k)} + t \big(\mu_\ast - \mu^{(k)}\big)\big] 
		\big\|_{{\mathcal H}_\infty}   -   \nabla^2 \big\| H\big[\mu^{(k)}\big] \big\|_{{\mathcal H}_\infty} \right]\big(\mu_\ast - \mu^{(k)}\big) \mathrm{d}t.
\end{split}							
\end{equation}
Now by exploiting the interpolation property, in particular part \textbf{(iv)} of Lemma~\ref{thm:basic_int}, 
and recalling $\boldsymbol{\omega}^{{\mathcal V}_k, {\mathcal W}_k}\big(\mu^{(k)}\big) = \omega^{(k)}$
due to part \textbf{(ii)} of Proposition~\ref{thm:localrep_smoothness}, we get
\begin{equation*} 
	\begin{split}
	\nabla \| H(\mu^{(k)}) \|_{{\mathcal H}_\infty} & =  \sigma^{ {\mathcal V}_k, {\mathcal W}_k }_\mu (\mu^{(k)},\omega^{(k)}) \\
						& =  \sigma^{ {\mathcal V}_k, {\mathcal W}_k }_\mu (\mu^{(k)}, \boldsymbol{\omega}^{{\mathcal V}_k, {\mathcal W}_k}\big(\mu^{(k)}\big)) \\
						& =  \sigma^{ {\mathcal V}_k, {\mathcal W}_k }_\mu (\mu^{(k)}, \boldsymbol{\omega}^{{\mathcal V}_k, {\mathcal W}_k}\big(\mu^{(k)}\big))
					+ \sigma^{ {\mathcal V}_k, {\mathcal W}_k }_\omega (\mu^{(k)}, \boldsymbol{\omega}^{{\mathcal V}_k, {\mathcal W}_k}\big(\mu^{(k)}\big))
						\nabla \boldsymbol{\omega}^{{\mathcal V}_k, {\mathcal W}_k}\big(\mu^{(k)}\big) \\
						& =  \nabla \widetilde{\sigma}^{ {\mathcal V}_k, {\mathcal W}_k } (\mu^{(k)}),
	\end{split}
\end{equation*}
where we employ 
$\sigma^{ {\mathcal V}_k, {\mathcal W}_k }_\omega (\mu^{(k)}, \boldsymbol{\omega}^{{\mathcal V}_k, {\mathcal W}_k}\big(\mu^{(k)}\big)) = 0$
for the third equality. Hence, equation \eqref{eq:Taylor_beg} can be rearranged as
\begin{equation}\label{eq:intermed}
\begin{split}
	0  & = \left[ \nabla^2 \widetilde{\sigma}^{{\mathcal V}_k, {\mathcal W}_k}\big(\mu^{(k)}\big) \right]^{-1} 
	\nabla \widetilde{\sigma}^{ {\mathcal V}_k, {\mathcal W}_k } \big(\mu^{(k)}\big)	+  
						\big(\mu_\ast - \mu^{(k)} \big)  +     \\ 
		& \hskip 5ex \left\{
		\left[ \nabla^2 \big\| H\big[\mu^{(k)}\big] \big\|_{{\mathcal H}_\infty} \right]^{-1}  - 
		\left[ \nabla^2 \widetilde{\sigma}^{{\mathcal V}_k, {\mathcal W}_k}\big(\mu^{(k)}\big) \right]^{-1}
		\right\}  \nabla \big\| H\big[\mu^{(k)}\big] \big\|_{{\mathcal H}_\infty} +  \\
				 & 	\left[ \nabla^2 \big\| H\big[\mu^{(k)}\big] \big\|_{{\mathcal H}_\infty} \right]^{-1} \\
		& \hskip 6ex \times \int_0^1	\left[ \nabla^2 \big\| H\big[\mu^{(k)} + t \big(\mu_\ast - \mu^{(k)}\big)\big] \big\|_{{\mathcal H}_\infty}   -   \nabla^2 \big\| H\big[\mu^{(k)}\big] \big\|_{{\mathcal H}_\infty} \right] \big(\mu_\ast - \mu^{(k)}\big)  \mathrm{d}t.
\end{split}					
\end{equation}

Throughout the rest of the proof, by manipulating \eqref{eq:intermed}, we bound $\| \mu^{(k+1)} - \mu_\ast \|_2$ from above in terms of 
$\| \mu^{(k)} - \mu_\ast \|_2$ and $\| \mu^{(k-1)} - \mu_\ast \|_2$.
Since $\widetilde{\omega}^{(k+1)} \in \overline{{\mathcal B}}(\omega_\ast,\eta_{\omega,0})$ is assumed to be the unique global maximizer
of ${\sigma}^{{\mathcal V}_k, {\mathcal W}_k}\big(\mu^{(k+1)}, \cdot\big)$, we must have 
$\boldsymbol{\omega}^{{\mathcal V}_k, {\mathcal W}_k}\big(\mu^{(k+1)}\big)$ $= \widetilde{\omega}^{(k+1)}$ by part \textbf{(iv)} 
of Proposition~\ref{thm:localrep_smoothness}. It follows that
\begin{equation*}
\nabla \widetilde{\sigma}^{ {\mathcal V}_k, {\mathcal W}_k } \big(\mu^{(k+1)}\big)  = \sigma^{ {\mathcal V}_k, {\mathcal W}_k }_\mu (\mu^{(k+1)},\widetilde{\omega}^{(k+1)}) = 
			\nabla \big\| H^{{\mathcal V}_{k}, {\mathcal W}_{k}} [\mu^{(k+1)}] \big\|_{{\mathcal H}_\infty} = 0,
\end{equation*}
where we use the fact that $\mu^{(k+1)}$ is a maximizer of $\left\| H^{{\mathcal V}_{k}, {\mathcal W}_{k}} [\cdot] \right\|_{{\mathcal H}_\infty}$
for the last equality. Moreover, a Taylor expansion yields
\begin{align*}
	0 &= \nabla \widetilde{\sigma}^{ {\mathcal V}_k, {\mathcal W}_k } \big(\mu^{(k+1)}\big)  \\
		  & =
		\nabla \widetilde{\sigma}^{ {\mathcal V}_k, {\mathcal W}_k }  \big(\mu^{(k)}\big)   + 
		\nabla^2 \widetilde{\sigma}^{ {\mathcal V}_k, {\mathcal W}_k } \big(\mu^{(k)}\big) \big(\mu^{(k+1)} - \mu^{(k)}\big)  + 
		\mathcal{O}\left( \big\| \mu^{(k+1)} - \mu^{(k)} \big\|_2^2 \right), 
\end{align*}
which in turn implies 
\begin{equation}\label{eq1}
 \left[ \nabla^2 \widetilde{\sigma}^{{\mathcal V}_k, {\mathcal W}_k}\big(\mu^{(k)}\big) \right]^{-1} 
	\nabla \widetilde{\sigma}^{ {\mathcal V}_k, {\mathcal W}_k } \big(\mu^{(k)}\big) = \big(\mu^{(k)} - \mu^{(k+1)}\big)  + \mathcal{O}\left( \big\| \mu^{(k+1)} - \mu^{(k)} \big\|_2^2 \right).
\end{equation}
%
%
%
Additionally, by another Taylor expansion,
\begin{align*}
 0 &=  \nabla \big\| H\big[\mu_*\big] \big\|_{{\mathcal H}_\infty}  \\ &=  \nabla \big\| H\big[\mu^{(k)}\big] \big\|_{{\mathcal H}_\infty}  +  \nabla^2 \big\| H\big[\mu^{(k)}\big] \big\|_{{\mathcal H}_\infty}  \big( \mu_* - \mu^{(k)} \big) + \mathcal{O}\left( \big\| \mu^{(k)} - \mu_* \big\|_2^2 \right).
\end{align*}
Therefore, by using Lemma~\ref{thm:acc_2der} and part \textbf{(i)} of Lemma~\ref{thm:3der_bounded}, we see that
\begin{equation} \label{eq2}
\begin{split}
& \left\|\left\{
		\left[ \nabla^2 \big\| H\big[\mu^{(k)}\big] \big\|_{{\mathcal H}_\infty} \right]^{-1}  - 
		\left[ \nabla^2 \widetilde{\sigma}^{{\mathcal V}_k, {\mathcal W}_k}\big(\mu^{(k)}\big) \right]^{-1}
		\right\} \cdot \nabla \big\| H\big[\mu^{(k)}\big] \big\|_{{\mathcal H}_\infty} \right\|_2 \\ 
& \quad \le \zeta \big\| \mu^{(k)} - \mu^{(k-1)} \big\|_2 \cdot \left\| \nabla \big\| H\big[\mu^{(k)}\big] \big\|_{{\mathcal H}_\infty} \right\|_2 
             = \mathcal{O} \left( \big\| \mu^{(k)} - \mu^{(k-1)} \big\|_2 \cdot \big\|\mu^{(k)} - \mu_*\big\|_2 \right).
\end{split}
\end{equation}
%
%
%
Finally, by exploiting the Lipschitz continuity of $\nabla^2 {\| H[\cdot] \|}_{{\mathcal H}_\infty}$ near $\mu_*$, we obtain
\begin{equation}\label{eq3}
\begin{split}
& \left\| \left[ \nabla^2 \big\| H\big[\mu^{(k)}\big] \big\|_{{\mathcal H}_\infty} \right]^{-1}  \right. \\ 
& \left.
  \quad \times \int_0^1	\left[ \nabla^2 \big\| H\big[\mu^{(k)} + t \big(\mu_\ast - \mu^{(k)}\big)\big] \big\|_{{\mathcal H}_\infty}   -   \nabla^2 \big\| H\big[\mu^{(k)}\big] \big\|_{{\mathcal H}_\infty} \right]   \big(\mu_\ast - \mu^{(k)}\big)  \mathrm{d}t \right\|_2 \\ 
& \hskip 50ex = \mathcal{O} \left( \big\|\mu^{(k)} - \mu_*\big\|_2^2 \right).
\end{split}
\end{equation}
%
%
Combining \eqref{eq:intermed} with \eqref{eq1}, \eqref{eq2}, \eqref{eq3}, and noting
$\big\| \mu^{(k)} - \mu^{(k-1)} \big\|_2 \leq 2 \max\big\{ \big\| \mu^{(k)} - \mu_\ast \big\|_2, \big\| \mu^{(k-1)} - \mu_\ast \big\|_2 \big\}$,  
we finally obtain
\[
	\big\| \mu^{(k+1)} - \mu_\ast \big\|_2 \leq	
	c_1 \max  \big\{ \big\| \mu^{(k)} - \mu_\ast \big\|_2, \big\| \mu^{(k-1)} - \mu_\ast \big\|_2 \big\} \big\| \mu^{(k)} - \mu_\ast \big\|_2 +  c_2 \big\| \mu^{(k)} - \mu_\ast  \big\|_2^2
\] 
for some constants $c_1,\, c_2$ independent of $\mu^{(1)}$ from which \eqref{eq:super_decay} is immediate.
\end{proof}

\begin{remark}
One important assumption for the rate of convergence result above is that the global minimizer $\mu_\ast$
is contained in the interior of $\underline{\Omega}$. Suppose $\underline{\Omega}$ is a box, and 
$\mu_\ast$ lies on the boundary of this box. Then one or more of the box constraints are active for the
full-order problem at $\mu_\ast$, and
$\| H[\cdot] \|_{{\mathcal H}_\infty}$ is increasing in all directions pointing into the interior of $\underline{\Omega}$
in a ball $\overline{{\mathcal B}}(\mu_\ast,\eta)$ (as $\| H[\cdot] \|_{{\mathcal H}_\infty}$ is continuously differentiable in a neighborhood of $\mu_\ast$). The same property holds to be true for the reduced function 
$\widetilde{\sigma}^{{\mathcal V}_k, {\mathcal W}_k}(\cdot)$ in another ball
$\overline{{\mathcal B}}(\mu_\ast,\widetilde{\eta}) \subseteq \overline{{\mathcal B}}(\mu_\ast,\eta)$, due to the interpolation
properties (specifically due to part \textbf{(iii)} of Lemma~\ref{thm:basic_int}), and uniform upper bounds on the
derivatives of $\sigma^{{\mathcal V}_k,{\mathcal W}_k}(\cdot, \cdot)$, $\widetilde{\sigma}^{{\mathcal V}_k,{\mathcal W}_k}(\cdot)$
(see in particular Lemma~\ref{thm:bound_hd} and \ref{thm:3der_bounded}), provided $\mu^{(1)}$ is close enough to $\mu_\ast$.
Consequently, the same active box constraints for the original function $\| H[\cdot] \|_{{\mathcal H}_\infty}$ at $\mu_\ast$ have to be 
active for the reduced function $\widetilde{\sigma}^{{\mathcal V}_k,{\mathcal W}_k}(\cdot)$ at $\mu^{(k+1)}$.
This means that the rate of convergence analysis above, in particular the proof of Theorem~\ref{thm:super_converge},
is applicable by restricting $\mu$ to the variables that are not active at $\mu_\ast$. If all of the constraints are active at $\mu_\ast$, then $\mu^{(k+1)} = \mu_\ast$ in exact arithmetic.

The minimizers for the examples arising from real applications on which we perform numerical experiments in the next section
turn out to be on the boundary of the box, see, e.\,g., Example~\ref{ex:T2DAL} where all of the three box constraints are active
at the minimizer, or Example~\ref{ex:SECM} where only one of the two box constraints is active, while the
other is inactive. On the other hand, the minimizer for the synthetic example in the next section is usually in the interior, see Example~\ref{ex:synthetic}.
\end{remark}

\section{Numerical Experiments}\label{sec:experiments}
In this section, we present numerical results obtained by our MATLAB implementation of Algorithm~\ref{alg1} that we made available for download. 
We first discuss some important implementation details and the test setup in the next subsection. Then, we report 
the numerical results on several large-scale linear parameter-dependent systems which we describe in detail. All test examples are taken from the Model Order Reduction Wiki (MOR Wiki) website\footnote{available at \url{https://morwiki.mpi-magdeburg.mpg.de/morwiki/index.php/Main_Page}.}. Our numerical experiments have been performed on a machine with an 4 Intel\textsuperscript{\textregistered} Core\textsuperscript{\texttrademark} i5-4590 CPUs with 3.30GHz each and 16GB RAM using Linux version 4.4.132-53-default and MATLAB version 9.4.0.813654 (R2018a).

\subsection{Implementation Details and Test Setup}
At each iteration of Algorithm~\ref{alg1}, the ${\mathcal L}_\infty$-norm of the transfer function of a reduced 
parametrized system needs to be minimized. 
We have implemented and tested two optimization techniques to solve this global non-convex optimization problem:
\begin{itemize}
 \item \texttt{eigopt}, a MATLAB implementation of the algorithm in \cite{Mengi2014}, which is an adaptation of the algorithm in \cite{Breiman1993} for eigenvalue optimization. This MATLAB package creates 
a lower and an upper bound for the optimal value of a given eigenvalue function by employing piece-wise quadratic 
support functions, and terminates when the difference between these bounds is less than a prescribed tolerance. 
For reliability and efficiency, one should supply an appropriate global lower bound $\gamma$ on the minimum eigenvalue of the Hessian of the eigenvalue function to be minimized to \texttt{eigopt}. 
 This solver can be slow, if there are many parameters or if $\gamma$ is very small. For our tests we always use $\gamma=-10000$.
 \item \texttt{GRANSO} \cite{Curtis2017}, which is based on BFGS together with line searches ensuring the satisfaction
of the weak Wolfe conditions. \texttt{GRANSO} converges to a locally optimal solution, that is not necessarily optimal
globally, but works efficiently even when there are several parameters.
\end{itemize}

Algorithm~\ref{alg1} is terminated in practice when the relative distance between $\mu^{(k)}$ and $\mu^{(k-1)}$ is less 
than a prescribed tolerance for some $k> 1$, if the minimal ${\mathcal L}_\infty$-norm values for the reduced transfer 
functions at two consecutive iterations differ by less than a prescribed tolerance, or if the number of iterations exceeds a specified integer, more formally, we terminate, if
\begin{multline*}
	k >  k_{\max} \quad \text{or} \quad \big\| \mu^{(k)} - \mu^{(k-1)} \big\|_2  <  \varepsilon_1 \cdot \frac{1}{2} \big\|\mu^{(k)} + \mu^{(k-1)}\big\|_2
		 \quad \text{or}  \\
	\left| \big\| H^{{\mathcal V}_k,{\mathcal W}_k}\big[\mu^{(k+1)}\big] \big\|_{{\mathcal H}_\infty}  -  
					\big\| H^{{\mathcal V}_{k-1},{\mathcal W}_{k-1}}\big[\mu^{(k)}\big] \big\|_{{\mathcal H}_\infty} \right|
					 < \\
			\varepsilon_2 \cdot \frac{1}{2} \left\{  \big\| H^{{\mathcal V}_k,{\mathcal W}_k}\big[\mu^{(k+1)}\big] \big\|_{{\mathcal H}_\infty}  +  
					\big\| H^{{\mathcal V}_{k-1},{\mathcal W}_{k-1}}\big[\mu^{(k)}\big] \big\|_{{\mathcal H}_\infty} \right\}.
\end{multline*}
In our numerical experiments, we set $\varepsilon_1 = \varepsilon_2 = 10^{-6}$ and $k_{\max} = 20$.

The absolute termination tolerance for the accuracy of the global optimizer computed by \texttt{eigopt} is $10^{-8}$, whereas the tolerance for reaching (approximate) stationarity in \texttt{GRANSO} is set to $10^{-12}$. Apart from these we use default options in \texttt{eigopt}, \texttt{GRANSO}, as well as our MATLAB routine \texttt{linorm\_subsp} that implements the method from \cite{Aliyev2017} for computing the $\mathcal{L}_\infty$-norm of the transfer function of a large-scale linear system. In \texttt{linorm\_subsp} we call the FORTRAN routine \texttt{AB13HD.F} via a mex file that implements the method of \cite{BenSV12} to compute the $\mathcal{L}_\infty$-norm of small-scale reduced systems. The latter is often faster and more reliable than the native MATLAB routine \texttt{norm} from the Control Systems Toolbox, that one could use for small-scale $\mathcal{L}_\infty$-norm computations as well. Our initial reduced order models are generated by 10 interpolation points (which consist of pairs of parameter values $\mu$ and frequencies $\omega$) that are equidistantly aligned on a line in $\underline{\Omega} \times [0,\omega_{\max})$ where $\omega_{\max}$ is a problem-dependent maximum frequency. Further details on the implementation can be inferred from the code that we have made available for download. 

\subsection{Results for Real Examples}\label{sec:real_exams}
We first test our algorithm on the following four parameter-dependent descriptor systems, all of which
originate from real applications.

\begin{example}[Thermal conduction (\texttt{T2DAL\_BCI}), see \cite{RudK05}]\label{ex:T2DAL}
Our first example is a thermal conduction model in a chip production. For a compact and efficient model of thermal conduction, one should take into account different configurations of the boundary conditions. This gives the capability to the chip producers to assess how the change in the environment influences the temperature in the chip. A mathematical model of the thermal conduction is given by the heat equation where the heat exchange through the three device interfaces is modeled by convection boundary conditions. These boundary conditions introduce the parameters $\mu_1,\,\mu_2,\,\mu_3$, called  the film coefficients, to describe the change in the temperature on the three device interfaces.
After spatial discretization of the partial differential equation and by incorporating the boundary conditions one obtains a time-invariant linear system with transfer function
\begin{align} \label{sys:thermal}
H[\mu_1,\mu_2,\mu_3](s) = C(sE - (A_0 + \mu_1 A_1 + \mu_2 A_2 + \mu_3 A_3))^{-1}B
\end{align}
where $E \in \R^{4257\times 4257}$ and $A_i \in {\mathbb R}^{4257 \times 4257}$, $i=1,\,2,\,3$ are diagonal matrices arising from the discretization of convection 
boundary conditions on the $i$-th interface and $B \in {\mathbb R}^{4257 \times 1}$, $C \in {\mathbb R}^{7 \times 4257}$ are the input and output matrices, respectively.
The specified box for the parameter $\mu := \left[\begin{smallmatrix} \mu_1 \\ \mu_2 \\ \mu_3 \end{smallmatrix}\right]$ is $\big[1,10^4\big] \times \big[1,10^4\big] \times \big[1,10^4\big]$.
\end{example}
We report on the results of Algorithm~\ref{alg1} on the \texttt{T2DAL\_BCI} example for different setups in Table~\ref{tab:t2dal}.
\begin{table}[htb]
\centering
\caption{Numerical results for the \texttt{T2DAL\_BCI} example.}   \label{tab:t2dal}
 \begin{tabular}{c|cccc}
    setup & $n_{\rm iter}$ & $(\mu_{1,*},\mu_{2,*},\mu_{3,*})$ & $\left\| H[\mu_{1,*},\mu_{2,*},\mu_{3,*}] \right\|_{\mathcal{H}_\infty} $  & time in s \\ \hline
    \texttt{eigopt} & 2 & (1.0000e+4,\,1.0000e+4,\,1.0000e+4) & 1.15429e+1 & 374.25 \\
    \texttt{GRANSO} & 2 & (1.0000e+4,\,1.0000e+4,\,1.0000e+4) & 1.15429e+1 & 2.54 
 \end{tabular}
\end{table}

\begin{example}[Anemometer (\texttt{anemometer\_1p} and \texttt{anemometer\_3p}), see  \cite{Greiner2011}]
An ane\-mo\-meter is a device to measure heat flow which consists of a heater and temperature sensors placed near the heater. The temperature field is affected by the flow and hence a temperature difference occurs between the sensors. 
The measured temperature difference determines the velocity of the fluid flow. The mathematical model for the anemometer is given by the 
convection-diffusion equation
$$
	\rho c \frac{\partial T}{\partial t} = \nabla (\kappa \nabla T) -\rho c v \nabla T + q',
$$
where $\rho$ denotes the mass density, $c$ is the specific heat, $\kappa$ is the thermal conductivity, $v$ is the fluid velocity, $T$ is the 
temperature, and $q'$ is the heat flow. A spatial discretization of the convection-diffusion equation above, for instance by the finite element method, yields a parametric linear system with the transfer function
\begin{align*}
	H[v](s) = C (sE -( A_1 + v (A_2 - A_1)))^{-1}B 
\end{align*}
which depends on only the fluid velocity $v \in [0,1]$; or a parametric system with the transfer function 
$$
	H[c,\kappa,v](s) = C (s (E_1 + c E_2) -( A_1 + \kappa A_2 + cvA_3)   )^{-1}B  
$$
where three parameters $c\in[0,1],\, \kappa \in [1,2],\, v\in [0.1,2]$ appear. 
The input and output matrices $B$ and $C$ above result from separating the spatial variables in $q'$. 
We refer to these one parameter and three parameter examples
as \texttt{anemometer\_1p} and  \texttt{anemometer\_3p}, respectively. In both cases, the order of the state space is 29008, there is a single input and a single output.
\end{example}
We report on the results of Algorithm~\ref{alg1} on the \texttt{anemometer\_1p} and  \texttt{anemometer\_3p} examples for different setups in Tables~\ref{tab:ane1} and~\ref{tab:ane3}, respectively.
\begin{table}[htb]
\centering
\caption{Numerical results for the \texttt{anemometer\_1p} example.}   \label{tab:ane1}
 \begin{tabular}{c|cccc}
    setup & $n_{\rm iter}$ & $v_*$ & $\left\| H[v_*] \right\|_{\mathcal{H}_\infty} $  & time in s \\ \hline
    \texttt{eigopt} & 6 & 0.0000 & 1.32274e--2 & 32.68 \\ 
    \texttt{GRANSO} & 6 & -0.0000 & 1.32274e--2 & 30.91 
 \end{tabular}
\end{table}

\begin{table}[htb]
\centering
\caption{Numerical results for the \texttt{anemometer\_3p} example.}   \label{tab:ane3}
 \begin{tabular}{c|cccc}
    setup & $n_{\rm iter}$ & $(c_{*},\kappa_{*},v_{*})$ & $\left\| H[c_{*},\kappa_{*},v_{*}] \right\|_{\mathcal{H}_\infty} $  & time in s \\ \hline
    \texttt{eigopt} & 4 & (0.0000, 2.0000, 1.0000e--1) & 1.64723e--3 & 766.06 \\
    \texttt{GRANSO} & 3 & (0.0000, 2.0000, 8.3855e--1) & 1.64723e--3 & 40.93 
 \end{tabular}
\end{table}


\begin{example}[Scanning electrochemical microscopy (\texttt{SECM}), see \cite{FenKRK06}]\label{ex:SECM} 
Scanning electrochemical microscopy is a technique to analyze the electrochemical behavior of species (in different states of matter) at their interface. This example considers the chemical reaction between two species on an electrode. The species transport is described by the second Fick's law which leads to two partial diffusion equations with appropriate boundary conditions. A spatial discretization together with a boundary control then leads to a linear-time invariant system whose transfer function is 
$$
  H[h_1,h_2](s) = C (sE - (h_1A_1 + h_2D_2 -A_3))^{-1}B,
$$
%
where $E,\,A_1,\,A_2,\,A_3\in {\mathbb R}^{16912\times 16912}$, $B\in {\mathbb R}^{16912\times 1}$, and $C\in {\mathbb R}^{5\times16912}$ and $h_1,\,h_2$ are the parameters of the problem. The experiment is performed in the box $\big[1, e^2\big] \times \big[1, e^2\big]$.
\end{example}
The results for the \texttt{SECM} example are summarized in Table~\ref{tab:secm}
\begin{table}[htb]
\centering
\caption{Numerical results for the \texttt{SECM} example.}   \label{tab:secm}
 \begin{tabular}{c|cccc}
    setup & $n_{\rm iter}$ & $(h_{1,*},h_{2,*})$ & $\left\| H[h_{1,*},h_{2,*}] \right\|_{\mathcal{H}_\infty} $  & time in s \\ \hline
    \texttt{eigopt} & 5 & (1.0000,\, 4.1944) & 1.85588 & 180.01 \\
    \texttt{GRANSO} & 5 & (1.0000,\, 4.2882) & 1.85583 & 20.51 
 \end{tabular}
\end{table}




 
In all examples, we observe superliner convergence in the final iterations. Specifically, for the \texttt{SECM} example, we report the errors with respect to the iteration number in Table~\ref{tab:pquad}. Four additional 
iterations after the construction of the initial reduced model suffice to find the minimal ${\mathcal H}_{\infty}$-norm with the specified relative tolerances.
For most examples, in particular the ones with more than one parameter, using \texttt{GRANSO} is significantly faster than \texttt{eigopt}. On the other hand, 
in contrast to \texttt{GRANSO}, \texttt{eigopt} returns the global minimizer for the reduced problems 
and thus sometimes yields more reliable results. In particular, due to the local convergence issue with \texttt{GRANSO}, rarely 
the subspace framework equipped with \texttt{GRANSO} does not converge to the global minimizer of the full problem,
while the one with \texttt{eigopt} does converge to the global minimizer of the full problem. 
This can for example be seen in the synthetic example discussed below.


\begin{table}[htb]
\caption{The minimizers for the reduced problems as well as the errors of the iterates of Algorithm~\ref{alg1} and the corresponding errors in the ${\mathcal H}_{\infty}$-norms 
are listed for the \texttt{SECM} example by using \texttt{GRANSO} for optimization. Here, the short-hands 
$f^{(k)} := \big\| H^{{\mathcal V}_k, {\mathcal W}_k}[\mu^{(k+1)}] \big\|_{{\mathcal L}_{\infty}}$ and 
$f_\ast := {\| H[\mu_\ast] \|}_{{\mathcal H}_{\infty}}$ are used.}
\label{tab:pquad}
\begin{center}
	\begin{tabular}{l|ccc}
	$k$ & $\mu^{(k+1)}$ & ${\big\|\mu^{(k+1)} - \mu_\ast \big\|}_2$ &   $\big|f^{(k)} - f_\ast \big| $ \\	
	\hline	
   	0    & (1.0000, 1.2088) &	3.08 &    2.61e--1	\\
   	1    & (1.6579, 7.3891) &  3.17  &   2.72e--4   \\
   	2    & (1.4761, 6.3522) & 	2.12   &  1.51e--4   \\
	3    & (1.0000, 4.2882) &   1.55e--9 & 1.24e--12 \\
	4    & (1.0000, 4.2882) &   $<$ 1e--12 & $<$ 1e--12
	\end{tabular}
\end{center}
\end{table}

To our knowledge, there is no reliable and efficient algorithm for large-scale ${\mathcal H}_\infty$-norm minimization 
in the literature which we can use for comparison purposes and verify the correctness of the results obtained. 
Instead, for each example above, we have computed the ${\mathcal H}_\infty$-norm of the system for various 
values of $\mu$ near the computed optimal parameter value $\mu_\ast$.
According to these computations, the optimal parameter values listed above seem to be at least locally optimal.   
For three of the examples, the plots of the ${\mathcal H}_\infty$-norm as a function of $\mu$ are illustrated 
in Figure \ref{fig:fctplots}

\begin{figure}[htb]
 \centering
 \subfloat[][Graph of $\mu_1 \mapsto \left\|H{[}\mu_1,\mu_{*,2},\mu_{*,3}{]} \right\|_{\mathcal{H}_\infty}$ for the \texttt{T2DAL\_BCI} example.]{
%
%
\begin{tikzpicture}

\begin{axis}[%
width=2.0in,
height=1.3in,
at={(1.882in,1.237in)},
scale only axis,
xmin=0,
xmax=10000,
xlabel style={font=\color{white!15!black}},
xlabel={$\mu_1$},
ymin=10,
ymax=70,
ylabel style={font=\color{white!15!black}},
ylabel={$\left\|H[\mu_1,\mu_{*,2},\mu_{*,3}] \right\|_{\mathcal{H}_\infty}$},
axis background/.style={fill=white},
title style={font=\bfseries},
title={}
]
\addplot [color=black, line width=1.0pt, forget plot]
  table[row sep=crcr]{%
0	65.915583094463\\
100	61.7826322060286\\
200	58.2424742162099\\
300	55.1686671036538\\
400	52.4689552003212\\
500	50.0743843902283\\
600	47.9323347574401\\
700	46.0019178327485\\
800	44.2508515125339\\
900	42.6532867185532\\
1000	41.1882635255659\\
1100	39.8385935209386\\
1200	38.5900368876044\\
1300	37.4306871591195\\
1400	36.350504819984\\
1500	35.3409592436723\\
1600	34.394750613198\\
1700	33.505591652742\\
1800	32.6680346197886\\
1900	31.8773329302396\\
2000	31.1293295488897\\
2100	30.4203662641859\\
2200	29.7472094011674\\
2300	29.1069885792425\\
2400	28.4971459015644\\
2500	27.9153935473054\\
2600	27.359678177275\\
2700	26.8281509002894\\
2800	26.3191418041117\\
2900	25.831138255733\\
3000	25.3627663304234\\
3100	24.9127748509178\\
3200	24.4800216147184\\
3300	24.0634614643714\\
3400	23.662135916097\\
3500	23.2751641125987\\
3600	22.9017349047417\\
3700	22.5410998991317\\
3800	22.1925673358674\\
3900	21.855496681636\\
4000	21.5292938412944\\
4100	21.2134069064941\\
4200	20.907322371094\\
4300	20.6105617538697\\
4400	20.3226785784797\\
4500	20.0432556658887\\
4600	19.7719027022875\\
4700	19.5082540497732\\
4800	19.2519667716708\\
4900	19.0027188479693\\
5000	18.76020755987\\
5100	18.524148024224\\
5200	18.2942718622903\\
5300	18.0703259882514\\
5400	17.8520715046897\\
5500	17.6392826944991\\
5600	17.4317460989271\\
5700	17.2292596736087\\
5800	17.0316320147273\\
5900	16.8386816482054\\
6000	16.6502363765442\\
6100	16.4661326773929\\
6200	16.2862151490241\\
6300	16.1103359990102\\
6400	15.9383545714962\\
6500	15.7701369099052\\
6600	15.6055553523342\\
6700	15.4444881560569\\
6800	15.2868191495945\\
6900	15.1324374089587\\
7000	14.981236957143\\
7100	14.8331164839846\\
7200	14.6879790856733\\
7300	14.5457320212719\\
7400	14.4062864860757\\
7500	14.2695573995207\\
7600	14.1354632069983\\
7700	14.0039256946215\\
7800	13.8748698153953\\
7900	13.7482235266637\\
8000	13.6239176375012\\
8100	13.501885665524\\
8200	13.3820637022755\\
8300	13.2643902869839\\
8400	13.1488062875301\\
8500	13.0352547885644\\
8600	12.9236809860883\\
8700	12.8140320881096\\
8800	12.7062572209986\\
8900	12.6003073410726\\
9000	12.4961351510391\\
9100	12.3936950212554\\
9200	12.2929429151178\\
9300	12.1938363185057\\
9400	12.0963341730892\\
9500	12.0003968131388\\
9600	11.9059859056234\\
9700	11.8130643934014\\
9800	11.7215964415312\\
9900	11.6315473860106\\
10000	11.5428836854236\\
};
\addplot [color=red, line width=1.0pt, only marks, mark size=3.0pt, mark=o, mark options={solid, red}, forget plot]
  table[row sep=crcr]{%
9999.99999999873	11.542883685417\\
};
\end{axis}
\end{tikzpicture}%
} \quad
  \subfloat[][Graph of $\mu_2 \mapsto \left\|H{[}\mu_{\ast,1},\mu_{2},\mu_{\ast,3}{]} \right\|_{\mathcal{H}_\infty}$ for the \texttt{T2DAL\_BCI} example.]{
%
%
\begin{tikzpicture}

\begin{axis}[%
width=2.0in,
height=1.3in,
at={(1.882in,1.237in)},
scale only axis,
xmin=0,
xmax=10000,
xlabel style={font=\color{white!15!black}},
xlabel={$\mu_2$},
ymin=11.5428,
ymax=11.5442,
ylabel style={font=\color{white!15!black},yshift=1ex},
ylabel={$\left\|H[\mu_{*,1},\mu_2,\mu_{*,3}]\right\|_{\mathcal{H}_\infty}$},
yticklabel style={/pgf/number format/precision=3},
ytick={11.543, 11.544},
axis background/.style={fill=white},
title style={font=\bfseries},
title={}
]
\addplot [color=black, line width=1.0pt, forget plot]
  table[row sep=crcr]{%
0	11.5441060332711\\
100	11.5440877517884\\
200	11.5440696555401\\
300	11.5440517415588\\
400	11.5440340069735\\
500	11.5440164489391\\
600	11.5439990646743\\
700	11.5439818515372\\
800	11.5439648068061\\
900	11.5439479279582\\
1000	11.5439312124297\\
1100	11.5439146577587\\
1200	11.5438982614955\\
1300	11.5438820213106\\
1400	11.5438659348564\\
1500	11.5438499998607\\
1600	11.5438342140977\\
1700	11.5438185753878\\
1800	11.5438030816027\\
1900	11.5437877306321\\
2000	11.543772520433\\
2100	11.5437574490083\\
2200	11.5437425143794\\
2300	11.5437277146026\\
2400	11.5437130477991\\
2500	11.5436985121354\\
2600	11.5436841057365\\
2700	11.5436698268706\\
2800	11.5436556737375\\
2900	11.5436416446711\\
3000	11.5436277379774\\
3100	11.543613951966\\
3200	11.5436002850427\\
3300	11.5435867356203\\
3400	11.5435733021259\\
3500	11.5435599830284\\
3600	11.5435467768234\\
3700	11.5435336820327\\
3800	11.5435206971973\\
3900	11.5435078209054\\
4000	11.5434950517518\\
4100	11.5434823883573\\
4200	11.5434698293726\\
4300	11.5434573734692\\
4400	11.5434450193522\\
4500	11.5434327657347\\
4600	11.543420611361\\
4700	11.5434085549843\\
4800	11.5433965953919\\
4900	11.5433847314098\\
5000	11.5433729618177\\
5100	11.5433612855233\\
5200	11.543349701338\\
5300	11.5433382081696\\
5400	11.5433268049019\\
5500	11.5433154904599\\
5600	11.5433042638109\\
5700	11.5432931238731\\
5800	11.5432820696325\\
5900	11.5432711000689\\
6000	11.5432602142351\\
6100	11.5432494110845\\
6200	11.5432386896734\\
6300	11.5432280490755\\
6400	11.5432174883505\\
6500	11.5432070065834\\
6600	11.54319660285\\
6700	11.5431862762785\\
6800	11.5431760259765\\
6900	11.5431658511107\\
7000	11.5431557508084\\
7100	11.5431457242486\\
7200	11.543135770597\\
7300	11.5431258890457\\
7400	11.5431160787846\\
7500	11.543106339064\\
7600	11.54309666907\\
7700	11.5430870680668\\
7800	11.5430775352657\\
7900	11.5430680699762\\
8000	11.5430586714498\\
8100	11.543049338951\\
8200	11.5430400717884\\
8300	11.5430308692503\\
8400	11.543021730658\\
8500	11.5430126553503\\
8600	11.5430036426259\\
8700	11.5429946918533\\
8800	11.5429858023582\\
8900	11.5429769735389\\
9000	11.542968204721\\
9100	11.5429594953117\\
9200	11.5429508446899\\
9300	11.5429422522542\\
9400	11.5429337174063\\
9500	11.5429252395533\\
9600	11.5429168181303\\
9700	11.54290845253\\
9800	11.5429001422438\\
9900	11.5428918866705\\
10000	11.5428836852825\\
};
\addplot [color=red, line width=1.0pt, only marks, mark size=3.0pt, mark=o, mark options={solid, red}, forget plot]
  table[row sep=crcr]{%
9999.99826817524	11.542883685417\\
};
\end{axis}
\end{tikzpicture}%
} \\
  \subfloat[][Graph of $v \mapsto  {\| H{[}v{]} \|}_{{\mathcal H}_\infty}$ for the \texttt{anemometer\_1p} example.]{
%
%
\begin{tikzpicture}

\begin{axis}[%
width=2.0in,
height=1.3in,
at={(1.882in,1.237in)},
scale only axis,
xmin=0,
xmax=1,
xlabel style={font=\color{white!15!black}},
xlabel={$v$},
ymin=0,
ymax=0.5,
ylabel style={font=\color{white!15!black}},
ylabel={${\| H[v] \|}_{{\mathcal H}_\infty}$},
axis background/.style={fill=white},
title style={font=\bfseries},
title={}
]
\addplot [color=black, line width=1.0pt, forget plot]
  table[row sep=crcr]{%
0	0.0132273832605601\\
0.025	0.0259537242099875\\
0.05	0.0400627477157661\\
0.075	0.0567642721752804\\
0.1	0.0768032159162271\\
0.125	0.100326069189744\\
0.15	0.126900471086064\\
0.175	0.155670463765753\\
0.2	0.185570635434202\\
0.225	0.215522630416488\\
0.25	0.244572751614412\\
0.275	0.271964298178549\\
0.3	0.297157736809828\\
0.325	0.319817208440476\\
0.35	0.33977961554283\\
0.375	0.357017775119385\\
0.4	0.37160456008341\\
0.425	0.383681530926906\\
0.45	0.393433340825017\\
0.475	0.401067922747215\\
0.5	0.406801833547023\\
0.525	0.410849889067955\\
0.55	0.413418201764244\\
0.575	0.41469981709722\\
0.6	0.414872271500524\\
0.625	0.414096527339609\\
0.65	0.412516861623968\\
0.675	0.410261388377377\\
0.7	0.407442978280641\\
0.725	0.404160404960914\\
0.75	0.400499597698144\\
0.775	0.396534918167279\\
0.8	0.392330406753277\\
0.825	0.387940964223147\\
0.85	0.383413448961248\\
0.875	0.37878768004705\\
0.9	0.374097343318967\\
0.925	0.369370802099802\\
0.95	0.364631817115866\\
0.975	0.359900181813481\\
1	0.3551922801254\\
};
\addplot [color=red, line width=1.0pt, only marks, mark size=3.0pt, mark=o, mark options={solid, red}, forget plot]
  table[row sep=crcr]{%
0	0.0132273832605641\\
};
\end{axis}
\end{tikzpicture}%
} \quad
  \subfloat[][Graph of $(\mu_1,\mu_2) \mapsto  {\| H{[}\mu_1,\mu_2 {]} \|}_{{\mathcal H}_\infty}$ for the \texttt{SECM} example.]{\input{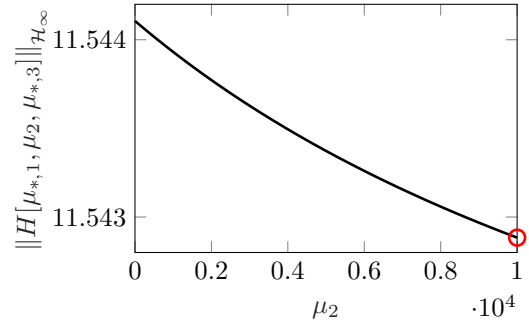}
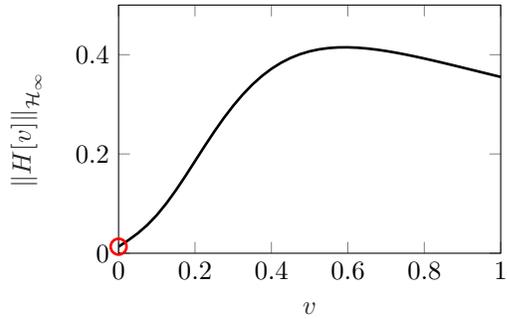
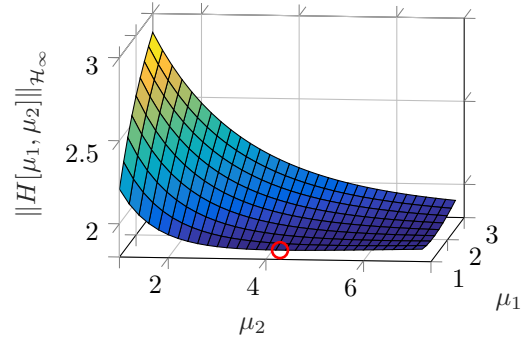
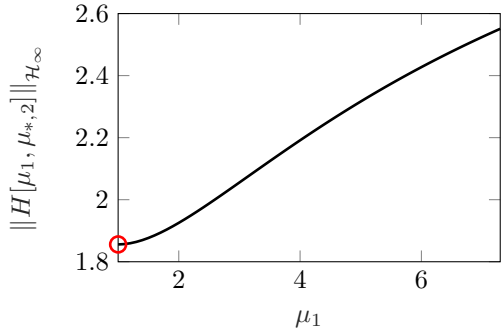
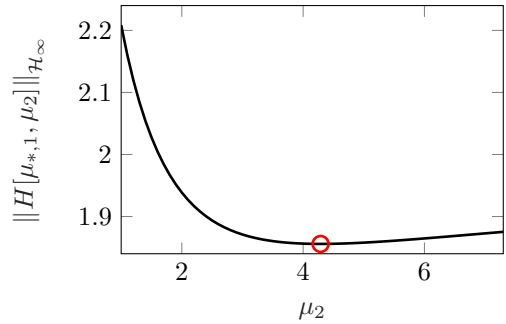} \\
  \subfloat[][Graph of $\mu_1 \mapsto  {\| H{[}\mu_1,\mu_{\ast,2} {]} \|}_{{\mathcal H}_\infty}$ for the \texttt{SECM} example.]{
%
%
\begin{tikzpicture}

\begin{axis}[%
width=2.0in,
height=1.3in,
at={(1.882in,1.237in)},
scale only axis,
xmin=1,
xmax=7.3,
xlabel style={font=\color{white!15!black}},
xlabel={$\mu_1$},
ymin=1.8,
ymax=2.6,
ylabel style={font=\color{white!15!black}},
ylabel={$\left\|H[\mu_1,\mu_{*,2}]\right\|_{\mathcal{H}_\infty}$},
axis background/.style={fill=white}
]
\addplot [color=black, line width=1.0pt, forget plot]
  table[row sep=crcr]{%
1	1.85582762017882\\
1.1	1.85684935455179\\
1.2	1.85964030489599\\
1.3	1.86400356989833\\
1.4	1.8697594277634\\
1.5	1.87674433789623\\
1.6	1.8848098996184\\
1.7	1.89382179321524\\
1.8	1.90365872499848\\
1.9	1.91421139412378\\
2	1.92538149499327\\
2.1	1.9370807654498\\
2.2	1.94923008777307\\
2.3	1.96175864677222\\
2.4	1.97460314705324\\
2.5	1.98770708978095\\
2.6	2.00102010791645\\
2.7	2.01449735792745\\
2.8	2.02809896528175\\
2.9	2.04178952059765\\
3	2.05553762306314\\
3.1	2.06931546764187\\
3.2	2.08309847258123\\
3.3	2.09686494382628\\
3.4	2.11059577308547\\
3.5	2.12427416646598\\
3.6	2.13788540080175\\
3.7	2.15141660499887\\
3.8	2.16485656394059\\
3.9	2.17819554269606\\
4	2.19142512897996\\
4.1	2.20453809199897\\
4.2	2.21752825599574\\
4.3	2.23039038696947\\
4.4	2.24312009120025\\
4.5	2.25571372434447\\
4.6	2.26816830999507\\
4.7	2.28048146671379\\
4.8	2.29265134264797\\
4.9	2.30467655693526\\
5	2.31655614718535\\
5.1	2.32828952240234\\
5.2	2.33987642078039\\
5.3	2.35131687186126\\
5.4	2.36261116260351\\
5.5	2.37375980695394\\
5.6	2.38476351856084\\
5.7	2.3956231863024\\
5.8	2.40633985234314\\
5.9	2.4169146924554\\
6	2.42734899837676\\
6.1	2.43764416199404\\
6.2	2.4478016611686\\
6.3	2.45782304703543\\
6.4	2.46770993262849\\
6.5	2.47746398269602\\
6.6	2.48708690458919\\
6.7	2.49658044011217\\
6.8	2.50594635824218\\
6.9	2.51518644862717\\
7	2.52430251578821\\
7.1	2.53329637395309\\
7.2	2.54216984246054\\
7.3	2.55092474167602\\
};
\addplot [color=red, line width=1.0pt, only marks, mark size=3.0pt, mark=o, mark options={solid, red}, forget plot]
  table[row sep=crcr]{%
1.00000000013407	1.8558276201789\\
};
\end{axis}
\end{tikzpicture}%
} \quad
  \subfloat[][Graph of $\mu_2 \mapsto  {\| H{[}\mu_{\ast,1},\mu_{2} {]} \|}_{{\mathcal H}_\infty}$ for the \texttt{SECM} example.]{
%
%
\begin{tikzpicture}

\begin{axis}[%
width=2.0in,
height=1.3in,
at={(1.882in,1.237in)},
scale only axis,
xmin=1,
xmax=7.3,
xlabel style={font=\color{white!15!black}},
xlabel={$\mu_2$},
ymin=1.84,
ymax=2.24,
ylabel style={font=\color{white!15!black}},
ylabel={$\left\|H[\mu_{*,1},\mu_2]\right\|_{\mathcal{H}_\infty}$},
axis background/.style={fill=white},
title style={font=\bfseries},
title={}
]
\addplot [color=black, line width=1.0pt, forget plot]
  table[row sep=crcr]{%
1	2.20855926584125\\
1.1	2.16097014869783\\
1.2	2.12000939476662\\
1.3	2.08466067220135\\
1.4	2.05408315538149\\
1.5	2.02757748438931\\
1.6	2.00455913375589\\
1.7	1.98453737954696\\
1.8	1.96709854942503\\
1.9	1.95189258557342\\
2	1.93862219503784\\
2.1	1.92703403820638\\
2.2	1.91691153491646\\
2.3	1.90806896305566\\
2.4	1.90034659608608\\
2.5	1.89360668018462\\
2.6	1.88773009326869\\
2.7	1.88261356031088\\
2.8	1.87816732437282\\
2.9	1.87431319242755\\
3	1.87098289054377\\
3.1	1.86811667530675\\
3.2	1.86566215818133\\
3.3	1.86357330738855\\
3.4	1.86180959821445\\
3.5	1.86033528779414\\
3.6	1.85911879457443\\
3.7	1.85813216604733\\
3.8	1.85735062111318\\
3.9	1.85675215570394\\
4	1.85631720216119\\
4.1	1.85602833440261\\
4.2	1.85587001218281\\
4.3	1.85582835880937\\
4.4	1.85589096755283\\
4.5	1.85604673271921\\
4.6	1.85628570196666\\
4.7	1.85659894695645\\
4.8	1.85697844986286\\
4.9	1.85741700362342\\
5	1.85790812412001\\
5.1	1.85844597273641\\
5.2	1.8590252879586\\
5.3	1.85964132486821\\
5.4	1.86028980153824\\
5.5	1.86096685147422\\
5.6	1.86166898135972\\
5.7	1.86239303346183\\
5.8	1.86313615213928\\
5.9	1.86389575396555\\
6	1.86466950104429\\
6.1	1.8654552771449\\
6.2	1.86625116633678\\
6.3	1.86705543383712\\
6.4	1.86786650882408\\
6.5	1.86868296899652\\
6.6	1.86950352668838\\
6.7	1.87032701636863\\
6.8	1.87115238337546\\
6.9	1.87197867375621\\
7	1.87280502509211\\
7.1	1.87363065820823\\
7.2	1.87445486967415\\
7.3	1.87527702501521\\
};
\addplot [color=red, line width=1.0pt, only marks, mark size=3.0pt, mark=o, mark options={solid, red}, forget plot]
  table[row sep=crcr]{%
4.28814375539346	1.8558276201789\\
};
\end{axis}
\end{tikzpicture}%
}
 \caption{The $\mathcal{\cH}_\infty$-norms for the different examples, where the computed minimal norm value is marked by a red circle.
 Note that in the captions and legends of (a)--(f), $\mu_{\ast,j}$ denotes the $j$th component of $\mu_\ast$ for $j = 1,\,2$, or 3.} 
 \label{fig:fctplots}
\end{figure}

\subsection{Results for Synthetic Examples}

Next, we test our approach on synthetic examples of various orders taken from the MOR Wiki. 

\begin{example}[Synthetic example]\label{ex:synth}\label{ex:synthetic}
We consider parametric single-input, single-output systems of order $\sn = 2\mathsf{q}$ with transfer functions of the form
\begin{equation}\label{eq:synth}
	H[\mu](s) = C(sI_\sn - \mu A_1 - A_0)^{-1}B,
\end{equation}
where the matrices $A_1,\,A_0\in \R^{\sn\times \sn}$, $B\in\R^{\sn\times 1}$, $C\in\R^{1\times \sn}$ are given by
\begin{align*}
	A_1 =
  		\begin{bmatrix}
    			A_{1,1} & & \\
    				    & \ddots & \\
    				    & & A_{1,\sm}
  		\end{bmatrix}, \;
  	A_0 =
  		\begin{bmatrix}
    			A_{0,1} & & \\
    				& \ddots & \\
    				& & A_{0,\sm}
  		\end{bmatrix}, \;
  	B =
  		\begin{bmatrix}
    			B_1\\
     			\vdots \\
    			B_{\sm}
  		\end{bmatrix}, \;
  	C =
  		\begin{bmatrix}
    			C_1 & \ldots &
    			C_{\sm}
  		\end{bmatrix}
\end{align*}
with 
$$ 
A_{1,i} =
  	\begin{bmatrix}
    		a_{i} & 0 \\
    		0 & a_{i}
  	\end{bmatrix}, \quad 
A_{0,i} =
  	\begin{bmatrix}
    		0 & b_{i}  \\
    		-b_i & 0
  	\end{bmatrix}, \quad  
  B_i =
  	\begin{bmatrix}
    		2\\
    		0
  	\end{bmatrix}, \quad
  C_i =
  	\begin{bmatrix}
    		1 &
    		0
  	\end{bmatrix}, \quad i=1,\,\ldots,\,\sm.
$$
The numbers $a_i$ and $b_i$ are chosen equidistantly in the intervals $[-10^3,-10]$ and $[10,10^3]$, respectively. 
The parameter $\mu$ is constrained to lie in the interval $[0.02,1]$.
\end{example}


We perform our experiments on this synthetic example for several values of $\sn$ varying in $10^2$, \ldots, $10^6$.
For smaller values of $\sn$, a comparison of Algorithm~\ref{alg1}
and the MATLAB package \texttt{eigopt} (for the unreduced problems) is provided in Table \ref{tab:synth1}.
This table indicates that with or without reduction we retrieve exactly the same optimal
${\mathcal H}_\infty$-norm values up to the prescribed tolerance $\varepsilon_2 = 10^{-6}$,
yet the proposed subspace framework leads to speed-ups on the order of $10^3$,
indeed the ratios of the runtimes increase quickly with respect to $\sn$.

Larger examples are considered in Table \ref{tab:synth2}, but only using the proposed subspace framework.
It does not seem possible to solve these larger ${\mathcal H}_\infty$-norm minimization problems in a reasonable 
time without employing reductions. Even the examples of order $10^6$ can be solved very fast. All examples 
up to order $10^5$ can be solved with just two to four iterations, only for very large examples up to 9 iterations 
may be needed. Moreover, the largest fraction of the computation time is spent for solving large-scale linear systems.   



Note that we have used \texttt{eigopt} for the optimization of the the small subproblems here which is guaranteed to return a global minimizer. We observe in practice that when the reduced problems are solved by a locally convergent algorithm, convergence to $\mu = 1$, a locally optimal solution, sometimes occurs. This is in particular the case for some values of $\sn$ when the reduced problems are solved with \texttt{GRANSO}. 
Also note that for the computation of the $\mathcal{L}_\infty$-norm in this example we make use of the native MATLAB function \texttt{norm}, since the periodic QZ algorithm used for the eigenvalue computation in \texttt{AB13HD.f} does not converge always. Further, we have set $\gamma = -1000$ in \texttt{eigopt} -- otherwise, the runtimes would be higher.

Finally, the progress of the subspace framework is displayed in Figure \ref{fig:synth_progress} on this synthetic
example for $\sn = 500$. After one subspace iteration, the ${\mathcal L}_\infty$-norm of the reduced problem
already closely resembles the one for the original problem around the minimizer. After two subspace iterations, it is even hard to
distinguish the ${\mathcal L}_\infty$-norm functions for the reduced and original problems around the minimizer, except for a thin peak that occurs in the reduced problem. The progress of the iteration is further summarized in Table~\ref{tab:prgsynthetic}.
\begin{table}[tb]
\caption{The minimizers for the reduced problems as well as the errors of the iterates of Algorithm~\ref{alg1} and the corresponding errors in the ${\mathcal H}_{\infty}$-norms 
are listed for the \texttt{synthetic} example for $\sn = 500$ by using \texttt{eigopt} for optimization. Here, again the short-hands 
$f^{(k)} := \big\| H^{{\mathcal V}_k, {\mathcal W}_k}[\mu^{(k+1)}] \big\|_{{\mathcal L}_{\infty}}$ and 
$f_\ast := {\| H[\mu_\ast] \|}_{{\mathcal H}_{\infty}}$ are used.}
\label{tab:prgsynthetic}
\begin{center}
	\begin{tabular}{l|ccc}
	$k$ & $\mu^{(k+1)}$ & ${\big| \mu^{(k+1)} - \mu_\ast \big|}_2$ &   $\big|f^{(k)} - f_\ast \big| $ \\	
	\hline	
   	0    & 0.5100 &	2.74e--1 &   0.12e--1	\\
   	1    & 0.2354 & 2.94e--4  &   5.01e--7   \\
   	2    & 0.2357 & $<$ 1e--12   &  $<$ 1e--12   \\
	\end{tabular}
\end{center}
\end{table}

\begin{table}[htb]
\centering
\caption{Results of the numerical experiments on Example~\ref{ex:synth} for smaller values of $\sn$, where we list
the number of subspace iterations $n_{\rm iter}$, the optimal parameter values by Algorithm~\ref{alg1} and \texttt{eigopt}, and the corresponding minimal ${\mathcal H}_{\infty}$-norms, as well as the runtimes are listed. The optimal ${\mathcal H}_\infty$-norm values returned by Algorithm~\ref{alg1} are the same with those returned by \texttt{eigopt} at least up to six decimal digits.}
\label{tab:synth1}
\begin{tabular}{cc|cc|cc|cc} 
  	  \multicolumn{2}{c}{}  & \multicolumn{2}{c}{$\mu_*$} &   \multicolumn{2}{c}{$\left\| H[\mu_*] \right\|_{{\mathcal H}_\infty}$}   & \multicolumn{2}{c}{runtime in s}   \\  
  	$\sn$  &  $n_{\rm iter}$ &  Alg.~\ref{alg1} &  \texttt{eigopt}  &  Alg.~\ref{alg1} & \texttt{eigopt}  & Alg.~\ref{alg1} & \texttt{eigopt} \\  \hline
	100    &   2 &  1.000000 & 1.000000  &  0.317092 &  0.317092 & 1.33 & 6.98    \\
	200    &  2 & 1.000000 &  1.000000 &  0.549800  & 0.549800  & 0.82 & 52.33 \\
	400    &  3 & 0.270587 &  0.270549 & 0.969289 & 0.969289  & 3.85 & 455.07 \\  
	600    & 4 & 0.212279 & 0.212255 & 1.337220 &   1.337219  & 3.06 & 1563.83 \\
	800   & 2 &  0.181492 & 0.181501 &  1.706940  &  1.706940 & 1.65 & 2635.76 \\
\end{tabular}
\end{table}

\begin{table}[htb]
\centering
\caption{
		The performance of Algorithm~\ref{alg1} on Example \ref{ex:synth} for larger values of $\sn$, where we have used \texttt{eigopt} for the optimization of the reduced subproblems. 
	}
\label{tab:synth2}
	\begin{tabular}{cc|ccc} 
		$\sn$		& $n_{\rm iter}$ &  $\mu_*$	& $\left\| H[\mu_*] \right\|_{\mathcal{H}_\infty}$ &	runtime in s	\\ \hline 
		1000    	& 4 & 0.157222 & 2.08316 &  3.61 \\
		2000    	& 4 & 0.115748 & 4.08243 &  4.68 \\
		5000    	& 2 & 0.113064 & 10.1718 &  2.15  \\
		10000    	& 2 & 0.112964 & 20.3321 &  1.64  \\
		20000    	& 2 & 0.113009 & 40.6554 &  1.37  \\
		50000    	& 2 & 0.113066 & 101.628 &  1.70 \\
		100000          & 2 & 0.113090 & 203.248 &  2.69  \\
		200000          & 2 & 0.113102 & 406.490 &  5.04  \\
		500000          & 2 & 0.113111 & 1016.22 & 12.53  \\
		1000000         & 2 & 0.113113 & 2032.43 & 26.11  \\
	\end{tabular}
\end{table}


\begin{figure}[htb]
\centering
\input{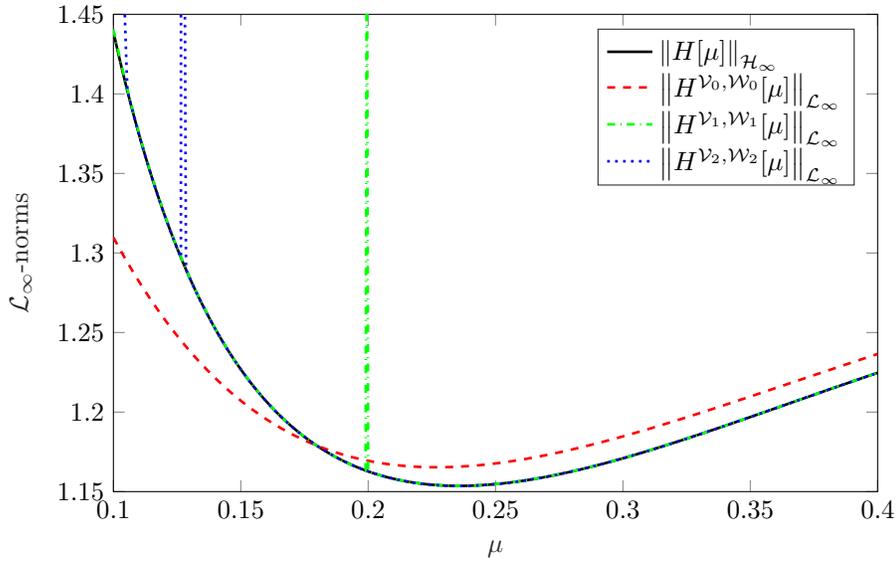}
	  \caption{ 
	  		The plots of the full function ${\| H[\cdot] \|}_{{\mathcal H}_\infty}$, as well as
			the reduced functions $\left\| H^{{\mathcal V}_0, {\mathcal W}_0}[\cdot] \right\|_{{\mathcal L}_\infty}$, $\left\| H^{{\mathcal V}_1, {\mathcal W}_1}[\cdot] \right\|_{{\mathcal L}_\infty}$,
			and $\big\| H^{{\mathcal V}_2, {\mathcal W}_2} [\cdot] \big\|_{{\mathcal L}_\infty}$
			in the interval $[0.1,0.4]$ for Example \ref{ex:synth} with $\sn = 500$.
	   	     }
	     \label{fig:synth_progress}
\end{figure}

\section{Concluding Remarks}
In this work we have developed new subspace restriction techniques to minimize the $\cH_\infty$-norm of transfer functions 
of large-scale parameter-dependent linear systems. We have given a detailed analysis of the rate of convergence of these 
methods, demonstrated the validity of the deduced rate of convergence results in practice by various numerical examples, 
which could all be solved very efficiently. The methods presented here make the design 
of optimal $\cH_\infty$-controllers for large-scale systems partly feasible. A fully feasible method to design optimal 
$\cH_\infty$-controllers for large-scale systems should also take stability considerations into account. We intend
to address stability issues in future.

\section*{Code Availability}
The MATLAB implementation of our algorithm and the computational results are publicly available under the URL \url{http://www.tu-berlin.de/?202212&L=1}.

\appendix

\section{Proof of Lemma~\ref{thm:bound_hd}}
By the continuity of $(\mu,\omega) \mapsto \sigma_{\min}(D(\mu,\omega))$,
	there exists a neighborhood $\widetilde{\mathcal N}$ of $(\mu_\ast,\omega_\ast)$ such that 
	$\sigma_{\min}(D(\mu,\omega)) \geq \beta/2$ for all $(\mu, \omega) \in \widetilde{\mathcal N}$.
	Consequently, the mapping $(\mu, \omega) \mapsto H[\mu]({\rm i} \omega)$ is continuously differentiable 
	and $\sigma(\cdot,\cdot),\, \sigma_2(\cdot,\cdot)$ are continuous in $\widetilde{\mathcal N}$. 
	The continuity of $\sigma(\cdot,\cdot),\, \sigma_2(\cdot,\cdot)$ 
	implies that $\sigma(\mu,\omega)$ remains a simple singular value of $H[\mu](\ri\omega)$, hence it is bounded away from zero in a neighborhood ${\mathcal N} \subseteq \widetilde{\mathcal N}$ of $(\mu_\ast, \omega_\ast)$. 
	Formally,
    \begin{equation}\label{eq:sval_gapl}
	\sigma(\mu, \omega) - \sigma_2(\mu, \omega)	  \geq   \widehat{\varepsilon}
	\quad	\forall (\mu, \omega) \in {\mathcal N}
    \end{equation}
for some $\widehat{\varepsilon} > 0$. 

Moreover, 
by employing the interpolation properties 
\begin{equation*}
	\sigma^{{\mathcal V}_k, {\mathcal W}_k}\big(\mu^{(k)},\omega^{(k)}\big) = 
			\sigma\big(\mu^{(k)},\omega^{(k)}\big)	\quad \text{and} \quad 
	\sigma^{{\mathcal V}_k, {\mathcal W}_k}_2\big(\mu^{(k)},\omega^{(k)}\big) = 
			\sigma_2\big(\mu^{(k)},\omega^{(k)}\big),
\end{equation*}
as well as the uniform Lipschitz continuity of $\sigma^{{\mathcal V}_k, {\mathcal W}_k}(\cdot,\cdot)$,
$\sigma^{{\mathcal V}_k, {\mathcal W}_k}_2(\cdot,\cdot)$
(i.\,e., parts \textbf{(iii)} and \textbf{(iv)} of Lemma~\ref{lemma:Lipschitz_cont}), there exists a region 
$\overline{{\mathcal B}}(\mu_\ast, \widetilde{\eta}_\mu) \times \overline{{\mathcal B}}(\omega_\ast, \widetilde{\eta}_\omega)$ 
in which $\sigma^{{\mathcal V}_k, {\mathcal W}_k}(\mu,\omega)$ is a simple, hence also a positive
singular value of $H^{{\mathcal V}_k, {\mathcal W}_k}[\mu](\ri \omega)$.
More precisely, we have
\begin{equation}\label{eq:sval_gap2l}
	\sigma^{{\mathcal V}_k, {\mathcal W}_k}(\mu, \omega) - \sigma^{{\mathcal V}_k, {\mathcal W}_k}_2(\mu, \omega)	 
							\geq  \varepsilon
	\quad	\forall (\mu, \omega) \in 
			\overline{{\mathcal B}}(\mu_\ast, \widetilde{\eta}_\mu) \times \overline{{\mathcal B}}(\omega_\ast, \widetilde{\eta}_\omega)
\end{equation}
for some $\varepsilon \in (0, \widehat{\varepsilon} )$, where the constants
$\varepsilon,\, \widetilde{\eta}_\mu,\, \widetilde{\eta}_\omega$ do not depend on $\mu^{(1)}$.
However, here it is assumed that ${\| \mu^{(1)} - \mu_\ast \|}_2$ is small enough in order to ensure
${\| \mu^{(k)} - \mu_\ast \|}_2 \ll \widehat{\varepsilon}$.

Let us now prove that $| \sigma_\omega(\cdot, \cdot) |$ and $| \sigma_{\omega \mu_1}(\cdot, \cdot) |$
are bounded from above uniformly in a neighborhood of $(\mu_\ast, \omega_\ast)$.
Our approach is based on the analytic continuation of the mapping
\[
	(\mu, s) \; \mapsto \; 
			\left[
				\begin{array}{cc}
								0	&	H^{{\mathcal V}_k, {\mathcal W}_k}[\mu](s)	\\
								H^{{\mathcal V}_k, {\mathcal W}_k}_\ast[\mu](s)	&	0	\\
				\end{array}
			\right]
				=:
			M^{{\mathcal V}_k, {\mathcal W}_k}[\mu](s)
\]
into the complex plane for $(\mu,s) \in \C^\sd \times \C$ near $(\mu_*, \ri \omega_*)$, where 
\begin{equation*}
	\begin{split}
	H^{\cV_k,\cW_k}_\ast[\mu](s)& := B^{W_k}_\ast(\mu) D^{V_k,W_k}_\ast(\mu,s)^{-1} C^{V_k}_\ast(\mu) 
	\quad \text{with} \\
	D^{V_k,W_k}_\ast(\mu,s) & := -s E^{V_k,W_k}_\ast(\mu)-A^{V_k,W_k}_\ast(\mu),
	\end{split}
\end{equation*}
and
\begin{equation*}
  \begin{split}
		E^{V_k,W_k}_\ast(\mu) & := f_1(\mu) ( W^*_k E_1 V_k )^\ast + \ldots + f_{\kappa_E}(\mu) ( W^*_k E_{\kappa_E} V_k )^\ast, \\
		A^{V_k,W_k}_\ast(\mu) & := g_1(\mu) ( W^*_k A_1 V_k )^\ast + \ldots + g_{\kappa_A}(\mu) ( W^*_k A_{\kappa_A} V_k )^\ast, \\
		B^{W_k}_\ast(\mu) & := h_1(\mu) ( W^*_k B_1 )^\ast + \ldots + h_{\kappa_B}(\mu) ( W^*_k B_{\kappa_B} )^\ast, \\
		C^{V_k}_\ast(\mu) & := k_1(\mu) ( C_1 V_k )^\ast + \ldots + k_{\kappa_C}(\mu) ( C_{\kappa_C} V_k )^\ast.
	\end{split}
\end{equation*}
Note that $\sigma^{{\mathcal V}_k, {\mathcal W}_k} (\mu, \omega)$ and  
$\sigma^{{\mathcal V}_k, {\mathcal W}_k}_2(\mu, \omega)$ correspond to the largest and second largest
eigenvalues of $M^{{\mathcal V}_k, {\mathcal W}_k}[\mu]({\rm i}\omega)$ for real $\omega$, as indeed
$H^{\cV_k,\cW_k}_\ast[\mu]({\rm i} \omega) = \big\{ H^{{\mathcal V}_k, {\mathcal W}_k}[\mu]({\rm i}\omega) \big\}^\ast$.
These Hermiticity properties are lost, when we replace $f_j,\, g_j,\, h_j,\, k_j$ with their analytic continuations 
$\widehat{f}_j,\, \widehat{g}_j,\, \widehat{h}_j,\, \widehat{k}_j$ or if we choose $s  \not\in \ri\R := \{ \ri \omega \: | \: \omega \in {\mathbb R} \}$. 
Let us denote the resulting extensions of $H^{{\mathcal V}_k, {\mathcal W}_k}$,
$H^{{\mathcal V}_k, {\mathcal W}_k}_\ast$, $M^{{\mathcal V}_k, {\mathcal W}_k}$
with $\widehat{H}^{{\mathcal V}_k, {\mathcal W}_k}$,
$\widehat{H}^{{\mathcal V}_k, {\mathcal W}_k}_\ast$, $\widehat{M}^{{\mathcal V}_k, {\mathcal W}_k}$.
As the subsequent arguments are for these complex continuations, in the rest of the proof
 $\overline{{\mathcal B}}_{\C}(\mu_\ast,\eta) := \{ \mu \in {\mathbb C}^\sd \; | \; {\| \mu - \mu_\ast \|}_2 \leq \eta \} 
     \quad \text{and} \quad 
 \overline{{\mathcal B}}_{\C}(\ri \omega_\ast,\eta) := \{ s \in {\mathbb C} \; | \; |s - \ri\omega_\ast | \leq \eta \}$
now denote the balls in the complex Euclidean spaces for a given radius $\eta > 0$.  
It is straightforward to verify that the uniform Lipschitz continuity of 
$(\mu, \omega) \mapsto H^{{\mathcal V}_k, {\mathcal W}_k}[\mu](\ri\omega)$ established in parts \textbf{(i)} and \textbf{(ii)} of
Lemma~\ref{lemma:Lipschitz_cont} extend to its complex counter-part, in particular, there exist $\gamma$, 
$\widehat{\eta}_\mu$, $\widehat{\eta}_\omega$
which are independent of $\mu^{(1)}$ such that
\begin{multline*}
	\big\| \widehat{H}^{{\mathcal V}_k, {\mathcal W}_k}\big[\widetilde{\mu}\big]\big(\widetilde{s}\big) - 
 \widehat{H}^{{\mathcal V}_k, {\mathcal W}_k}[\mu](s) \big\|_2   \\
 		\leq \big\| \widehat{H}^{{\mathcal V}_k, {\mathcal W}_k}\big[\widetilde{\mu}\big]\big(\widetilde{s}\big) - 
 \widehat{H}^{{\mathcal V}_k, {\mathcal W}_k}[\mu]\big(\widetilde{s}\big) \big\|_2
 				+
		\big\| \widehat{H}^{{\mathcal V}_k, {\mathcal W}_k}[\mu]\big(\widetilde{s}\big) - 
 \widehat{H}^{{\mathcal V}_k, {\mathcal W}_k}[\mu](s) \big\|_2	\\	
		 \leq  \gamma \big( \big\| \widetilde{s} - s \big\|_2 + \big| \widetilde{\omega} - \omega \big| \big)
\end{multline*}
for all $\widetilde{\mu},\, \mu \in \overline{{\mathcal B}}_{\C}(\mu_\ast,\widehat{\eta}_\mu) \subset {\mathbb C}^\sd$, and 
for all $\widetilde{s},\, s \in \overline{{\mathcal B}}_{\C}(\ri\omega_\ast,\widehat{\eta}_\omega) \subset {\mathbb C}$.
Analogous uniform Lipschitz continuity assertion also holds for $\widehat{H}^{{\mathcal V}_k, {\mathcal W}_k}_\ast$.
Consequently, there exist $\gamma$, $\widehat{\eta}_\mu$, $\widehat{\eta}_\omega$ which are independent of $\mu^{(1)}$ 
such that
\begin{multline}\label{eq:M_uniform_Lipschitz}
	\big\| \widehat{M}^{{\mathcal V}_k, {\mathcal W}_k}\big[\widetilde{\mu}\big]\big(\widetilde{s}\big) - 
			 \widehat{M}^{{\mathcal V}_k, {\mathcal W}_k}[\mu](s) \big\|_2	
			 \leq	 \gamma \big( \big\| \widetilde{\mu} - \mu \big\|_2 + \big| \widetilde{s} - s \big| \big)	 \\
	\forall\, \widetilde{\mu},\, \mu \in \overline{{\mathcal B}}_{\C}(\mu_\ast,\widehat{\eta}_\mu) \subset {\mathbb C}^\sd, \quad
	\forall\, \widetilde{s},\, s \in \overline{{\mathcal B}}_{\C}(\ri\omega_\ast,\widehat{\eta}_\omega) \subset {\mathbb C}.
\end{multline}

Now, for $(\mu,s) \in  \overline{{\mathcal B}}_{\C}(\mu_\ast,\widehat{\eta}_\omega) \times \overline{{\mathcal B}}_{\C}(\ri\omega_\ast,\widehat{\eta}_\omega)$, 
let us consider the eigenvalue $\widehat{\sigma}^{{\mathcal V}_k, {\mathcal W}_k}(\mu,s)$ 
of $\widehat{M}^{{\mathcal V}_k, {\mathcal W}_k}[\mu](s)$
corresponding to the eigenvalue $\sigma^{{\mathcal V}_k, {\mathcal W}_k}(\mu,\omega)$ 
of $M^{{\mathcal V}_k, {\mathcal W}_k}[\mu](\ri\omega)$, that is, $\widehat{\sigma}^{{\mathcal V}_k, {\mathcal W}_k}(\cdot,\cdot)$ is obtained by the analytic continuation of ${\sigma}^{{\mathcal V}_k, {\mathcal W}_k}(\cdot,\cdot)$ into the complex plane.  
This eigenvalue function is no more real-valued, since $\widehat{M}^{{\mathcal V}_k, {\mathcal W}_k}[\mu](s)$
is not necessarily a Hermitian matrix. However, by \eqref{eq:sval_gap2l} and \eqref{eq:M_uniform_Lipschitz},
as well as Theorem~5.1 in \cite[Chapter 4]{Stewart1990}, there exist $\eta_{\mu, m} \leq \min \big\{ \widetilde{\eta}_\mu, \widehat{\eta}_\mu \big\}$
and $\eta_{\omega,m} \leq \min \big\{ \widetilde{\eta}_\omega, \widehat{\eta}_\omega \big\}$
such that the eigenvalue $\widehat{\sigma}^{{\mathcal V}_k, {\mathcal W}_k}(\mu,s)$ 
remains simple for all $\mu \in \overline{{\mathcal B}}_{\C}(\mu_\ast,\eta_{\mu,m})$ and all 
$s \in {\mathcal B_{\C}}(\ri\omega_\ast, \eta_{\omega,m})$.
We remark that $\eta_{\mu,m}$ and $\eta_{\omega,m}$ are independent of ${\mathcal V}_k,\, {\mathcal W}_k$ and 
hence are independent of $\mu^{(1)}$.   
Now let us consider any $\eta_\mu \in (0, \eta_{\mu,m})$ and any $\eta_\omega \in (0, \eta_{\omega,m})$. By the analyticity of 
$\widehat{\sigma}^{{\mathcal V}_k, {\mathcal W}_k}(\cdot,\cdot)$ in the interior of 
$\overline{{\mathcal B}}_{\C}(\mu_\ast,\eta_{\mu,m}) \times \overline{{\mathcal B}}_{\C}(\ri\omega_\ast, \eta_{\omega,m})$,
for a given $\widetilde{\mu} \in \overline{{\mathcal B}}_{\C}(\mu_\ast,\eta_\mu)$ and 
$\widetilde{s} \in \overline{{\mathcal B}}_{\C}(\ri\omega_\ast, \eta_{\omega}/2)$, by Cauchy's integral formula we have
\begin{equation}\label{eq:Cauchy_int1}
	\widehat{\sigma}^{{\mathcal V}_k, {\mathcal W}_k}_{s}(\widetilde{\mu}, \widetilde{s})
			=\frac{1}{2\pi {\rm i}} \oint_{\big|s-\widetilde{s}\big| = \eta_\omega/2}
	\frac{\widehat{\sigma}^{{\mathcal V}_k, {\mathcal W}_k}(\widetilde{\mu},s)}{(s-\widetilde{s})^2} \mathrm{d}s.
\end{equation}
We claim that the term $\widehat{\sigma}^{{\mathcal V}_k, {\mathcal W}_k}(\widetilde{\mu},s)$ inside the integral
in modulus is uniformly bounded from above. To this end, as  
$\big| \widehat{\sigma}^{{\mathcal V}_k, {\mathcal W}_k}\big(\widetilde{\mu},s\big) \big| \leq \big\| \widehat{M}^{{\mathcal V}_k, {\mathcal W}_k}\big[\widetilde{\mu}\big](s) \big\|_2$, it suffices to show
the uniform boundedness of $\big\| \widehat{M}^{{\mathcal V}_k, {\mathcal W}_k}[\widetilde{\mu}](s) \big\|_2$.
Letting $\beta := \sigma_{\min}(D(\mu_\ast, \omega_\ast))$ and following the arguments at the beginning of the 
proof of Lemma~\ref{lemma:Lipschitz_cont}, there exists a neighborhood $\widehat{{\mathcal N}} \subset {\mathbb C}^\sd \times {\mathbb C}$ 
of $(\mu_\ast, \ri\omega_\ast)$ such that $\sigma_{\min}\big(D^{V_k,W_k} (\mu,s)\big) \geq \beta/2$ for all 
$(\mu, s) \in \widehat{{\mathcal N}}$. Without loss of generality, we assume 
$\widehat{\mathcal N} = \overline{{\mathcal B}}_{\C}(\mu_\ast,\eta_\mu) \times \overline{{\mathcal B}}_{\C}(\ri\omega_\ast,\eta_{\omega})$
(as we can choose $\eta_\mu$ and $\eta_\omega$ as small as we wish). Hence,
\begin{multline*}
	\big\| \widehat{H}^{{\mathcal V}_k, {\mathcal W}_k}[\mu](s) \big\|_2  \leq 2 \frac{ \big\| C^{V_k}(\mu) \big\|_2 \big\| B^{W_k}(\mu) \big\|_2 }{\beta}
						\leq 2 \frac{M_C M_B}{\beta}
	\\ \forall\, \mu \in \overline{{\mathcal B}}_{\C}(\mu_\ast,\eta_\mu),\; \forall\, s \in \overline{{\mathcal B}}_{\C}(\ri\omega_\ast,\eta_{\omega}),
\end{multline*}
where $M_{C} := \max \left\{ \| C(\mu) \|_2  \; | \; 
				\mu \in \overline{{\mathcal B}}_{\C}(\mu_\ast,\eta_\mu) \right\}$,
	  $M_{B} := \max \left\{ \| B(\mu) \|_2  \; | \; 
				\mu \in \overline{{\mathcal B}}_{\C}(\mu_\ast,\eta_\mu) \right\}$.
In an analogous fashion, the same upper bound also holds uniformly for 
$\big\| \widehat{H}^{{\mathcal V}_k, {\mathcal W}_k}_\ast[\mu](s) \big\|_2$ 		
for all $\mu \in \overline{{\mathcal B}}_{\C}(\mu_\ast,\eta_\mu)$ and all $s \in \overline{{\mathcal B}}_{\C}(\ri\omega_\ast,\eta_{\omega})$,
which gives rise to
\[
	\big\| \widehat{M}^{{\mathcal V}_k, {\mathcal W}_k}[\mu](s) \big\|_2  \leq  
			2 \frac{M_C M_B}{\beta} =: M
		\quad \forall\, \mu \in \overline{{\mathcal B}}_{\C}(\mu_\ast,\eta_\mu),\; \forall\,s \in \overline{{\mathcal B}}_{\C}(\ri\omega_\ast,\eta_{\omega}).
\]
We deduce from \eqref{eq:Cauchy_int1} that
\begin{multline*}
	\big| \widehat{\sigma}^{{\mathcal V}_k, {\mathcal W}_k}_{s}\big(\widetilde{\mu}, \widetilde{s}\big) \big|
			\leq
	\frac{1}{2\pi} \left\{ \max_{\big| s - \widetilde{s} \big| = \eta_{\omega}/2} \big|\widehat{\sigma}^{{\mathcal V}_k, {\mathcal W}_k}(\widetilde{\mu},s)\big| \right\}
	                              \frac{1}{(\eta_{\omega}/2)^2} (2 \pi \eta_{\omega}/2) \leq  \frac{2M}{\eta_\omega}   \\
	\forall\, \widetilde{\mu} \in \overline{{\mathcal B}}_{\C}(\mu_\ast,\eta_\mu), \;
	\forall\, \widetilde{s} \in \overline{{\mathcal B}}_{\C}(\ri\omega_\ast, \eta_{\omega}/2),                              
\end{multline*}
hence also $\big| {\sigma}^{{\mathcal V}_k, {\mathcal W}_k}_{\omega}\big(\widetilde{\mu}, \widetilde{\omega}\big) \big| \le 2M/\eta_\omega$ for all $\widetilde{\mu} \in \overline{{\mathcal B}}(\mu_\ast,\eta_\mu)$ and all $\widetilde{\omega} \in \overline{{\mathcal B}}(\omega_\ast, \eta_{\omega}/2)$.

Now let us consider the mixed derivative $\sigma_{s \mu_1}^{{\mathcal V}_k, {\mathcal W}_k}$,
specifically for a given $\widehat{\mu} \in \overline{{\mathcal B}}_{\C}(\mu_\ast,\eta_\mu/2)$ and 
$\widehat{s} \in \overline{{\mathcal B}}_{\C}(\ri\omega_\ast, \eta_{\omega}/2)$, we have
\begin{equation}\label{eq:Cauchy_int2}
	\widehat{\sigma}^{{\mathcal V}_k, {\mathcal W}_k}_{s\mu_1}\big(\widehat{\mu}, \widehat{s}\big)
			=\frac{1}{2\pi {\rm i}} \oint_{\mathcal C}
	\frac{\widehat{\sigma}^{{\mathcal V}_k, {\mathcal W}_k}_s\big(\mu,\widehat{s}\big)}{\big(\mu_1-\widehat{\mu}_1\big)^2} \mathrm{d}\mu_1,
\end{equation}
where the contour integral is over 
${\mathcal C} := \big\{ \mu \in {\mathbb C}^\sd \; \big| \; \big|\mu_1-\widehat{\mu}_1 \big| =\eta_\mu/2,\, \mu_j = \widehat{\mu}_j, \, j = 2,\dots, \sd \big\}$.
Taking the modulus of both sides in \eqref{eq:Cauchy_int2} yields
\begin{multline*}
	\big| \widehat{\sigma}^{{\mathcal V}_k, {\mathcal W}_k}_{s\mu_1}(\widehat{\mu}, \widehat{s}) \big|
		\leq
	\frac{1}{2\pi} \left\{ \max_{\mu \in {\mathcal C}} \big| \widehat{\sigma}^{{\mathcal V}_k, {\mathcal W}_k}_s(\mu,\widehat{s}) \big| \right\}
	 \frac{1}{(\eta_\mu/2)^2} (2 \pi \eta_\mu / 2)  \leq \frac{4M}{\eta_\mu \eta_\omega} 	\\
	\forall\, \widehat{\mu} \in \overline{{\mathcal B}}_{\C}(\mu_\ast,\eta_\mu/2), \;
	\forall\, \widehat{s} \in \overline{{\mathcal B}}_{\C}(\ri\omega_\ast, \eta_{\omega}/2).
\end{multline*}
The arguments above prove the uniform boundedness of 
$\big| \sigma^{{\mathcal V}_k,{\mathcal W}_k}_\omega(\cdot, \cdot) \big|,\, \big| \sigma^{{\mathcal V}_k,{\mathcal W}_k}_{\omega \mu_1}(\cdot, \cdot) \big|$.
The uniform boundedness of all other first three derivatives can be proven similarly. \hfill $\square$

\bibliography{Hinfmin}

\end{document}